\documentclass[11pt,a4]{amsart}

\title{A Counterexample to the First Zassenhaus Conjecture}
\author{Florian Eisele}
\author{Leo Margolis}

\address{Florian Eisele \newline Department of Mathematics, City, University of London,
	Northampton Square, London EC1V 0HB, United Kingdom \newline E-mail address: {\tt Florian.Eisele@city.ac.uk}}
\address{Leo Margolis \newline Departamento de Matem\`aticas, Universidad de Murcia, 30100 Murcia, Spain \newline E-mail address: {\tt leo.margolis@um.es}}

\usepackage{amssymb}
\usepackage{amsthm}
\usepackage{amsmath}
\usepackage[all]{xy}
\usepackage{geometry}
\usepackage{newcent,microtype}
\usepackage[english]{babel}
\usepackage{mathtools}
\usepackage{amsthm}
\usepackage{graphicx}
\usepackage{enumitem}

\usepackage[bookmarks=false,colorlinks=true,linkcolor=black,citecolor=black,filecolor=black,urlcolor=black]{hyperref} 

\newcommand{\eps}{\varepsilon}
\newcommand{\Z}{\mathbb Z}
\newcommand{\Q}{\mathbb Q}
\newcommand{\C}{\mathbb C}
\newcommand{\N}{\mathbb N}
\newcommand{\F}{\mathbb F}
\newcommand{\Hom}{\operatorname{Hom}}
\newcommand{\End}{{\rm{End}}}

\newcommand{\Ker}{\operatorname{Ker}}
\newcommand{\Irr}{\operatorname{Irr}}
\newcommand{\Tr}{\operatorname{Tr}}

\newcommand{\Cl}{\operatorname{Cl}}

\newcommand{\nr}{{\rm nr}}
\newcommand{\Nr}{{\rm Nr}}
\newcommand{\UU}{\mathcal{U}}

\newcommand{\OfU}{}

\theoremstyle{theorem}
\newtheorem{defi}{Definition}[section]

\newtheorem{thm}[defi]{Theorem}
\newtheorem{lemma}[defi]{Lemma}
\newtheorem{corollary}[defi]{Corollary}
\newtheorem{notation}[defi]{Notation}
\newtheorem{prop}[defi]{Proposition}

\theoremstyle{definition}
\newtheorem{remark}[defi]{Remark}

\newcommand\blfootnote[1]{%
	\begingroup
	\renewcommand\thefootnote{}\footnote{#1}%
	\addtocounter{footnote}{-1}%
	\endgroup
}

\newtheoremstyle{named}{}{}{\itshape}{}{\bfseries}{.}{.5em}{\thmnote{#3}}
\theoremstyle{named}
\newtheorem*{namedtheorem}{Theorem}

\mathtoolsset{showonlyrefs}

\begin{document}
	
\maketitle
\begin{abstract} Hans~J.~Zassenhaus conjectured that for any unit $u$ of finite order in the integral group ring of a finite group $G$ there exists a unit $a$ in the rational group algebra of $G$ such that $a^{-1}\cdot u \cdot a=\pm g$ for some $g\in G$. We disprove this conjecture by first proving general results that help identify counterexamples and then providing an infinite number of examples where these results apply.  Our smallest example is a metabelian group of order $2^7 \cdot 3^2 \cdot 5 \cdot 7^2 \cdot 19^2$ whose integral group ring contains a unit of order $7 \cdot 19$ which, in the rational group algebra, is not conjugate to any element of the form $\pm g$.
\end{abstract}

\blfootnote{The first author was supported by the EPSRC, grant EP/M02525X/1.  The second author was supported by a Marie Curie Individual Fellowship from EU project 705112-ZC and the FWO (Research Foundation Flanders).}

\section{Introduction}

Let $G$ be a finite group and denote by $RG$ the group ring of $G$ over a commutative ring $R$. Denote by $\UU(RG)$ the unit group of $RG$. 
In the 1970's Zassenhaus made three strong conjectures about finite subgroups of $\UU(\Z G)$ (cf. \cite[Section 37]{SehgalBook}). These conjectures, often called the first, second and third Zassenhaus conjecture and sometimes abbreviated as (ZC1), (ZC2) and (ZC3), had a lasting impact on research in the field. All three of these conjectures turned out to be true for nilpotent groups \cite{Weiss91}, but metabelian counterexamples for the second and the third one were found by K.~W.~Roggenkamp and L.~L.~Scott \cite{Scott, Klingler}. Later M.~Hertweck showed that there are counterexamples of order as small as $96$ \cite[Section 11]{HertweckHabil}.  
Unlike its siblings, the first Zassenhaus conjecture seemed to stand the test of time. Since it was the only one of the three to remain open, people in recent years started referring to it as \emph{the} Zassenhaus conjecture, and we will do the same in this article.
\begin{namedtheorem}[Zassenhaus Conjecture]
If $u\in \UU(\Z G)$ is a unit of finite order, then there is an $a\in \UU(\Q G)$ such that
$
a^{-1}\cdot u \cdot a = \pm g
$
for some $g\in G$.\\
\end{namedtheorem}

%A famous conjecture of Zassenhaus states:
%\begin{namedtheorem}[Zassenhaus Conjecture]
%If $u\in \UU(\Z G)$ is a unit of finite order, then there is an $a\in \UU(\Q G)$ such that
%$
%a^{-1}\cdot u \cdot a = \pm g
%$
%for some $g\in G$.\\
%\end{namedtheorem}
%This was the first of three strong conjectures about finite subgroups of $\UU(\Z G)$ made by Zassenhaus (cf. \cite[Section 37]{SehgalBook}). 
%These conjectures are often called the first, second and third Zassenhaus conjecture, sometimes abbreviated as (ZC1), (ZC2) and (ZC3).
%All three of these conjectures turned out to be true for nilpotent groups \cite{Weiss91}, but metabelian counterexamples for the second and the third one were found by K.~W.~Roggenkamp and L.~L.~Scott \cite{Scott, Klingler}. Later M.~Hertweck showed that there are counterexamples of order as small as $96$ \cite[Section 11]{HertweckHabil}.  
%Unlike its siblings, the first Zassenhaus conjecture seemed to stand the test of time. Since it was the only one of the three to remain open, people in recent years started referring to it as \emph{the} Zassenhaus conjecture, and we will do the same in this article.

This conjecture first appeared in written form in \cite{Zassenhaus} and inspired a lot of research in the decades to follow. The first results on the conjecture were mostly concerned with special classes of metabelian groups, \cite{HughesPearson72, PolcinoMilies73, AllenHobby80, BhandariLuthar83, RitterSehgal83, PolcinoMiliesSehgal84, Mitsuda86, PolcinoMiliesRitterSehgal86, SehgalWeiss86, MRSW87, LutharTrama90, LutharPassi92, LutharSehgal98, JuriaansMilies00, BovdiHoefertKimmerle04, delRioSehgal06}. Almost all of these results were later generalised by Hertweck \cite{HertweckColloq, HertweckEdinb}. Hertweck proved that the Zassenhaus conjecture holds for groups $G$ which have a normal Sylow $p$-subgroup with abelian complement or a cyclic normal subgroup $C$ such that $G=C\cdot A$ for some abelian subgroup $A$ of $G$. The latter result was further generalised in \cite{CyclicByAbelian}, proving that the Zassenhaus conjecture holds for cyclic-by-abelian groups. In a different vein, A.~Weiss' proof of the conjecture, or even a stronger version of it, for nilpotent groups \cite{Weiss88, Weiss91}, was certainly a highlight of the study. 
The conjecture is also known to hold for a few other classes of solvable groups \cite{Fernandes87, DokuchaevJuriaans96, BaechleKimmerleMargolisDFG, MargolisdelRioCW1, MargolisdelRioPAP, MargolisdelRioCW3}, as well as for some small groups. In particular, the conjecture holds for groups of order smaller than $144$ \cite{HoefertKimmerle06, HermanSingh15, SmallGroups}. 

Progress on non-solvable groups was initially lagging. For many years the conjecture was only known to hold for the alternating and symmetric group of degree $5$ \cite{LutharPassi, LutharTrama} and the special linear group $\operatorname{SL}(2,5)$ \cite{DokuchaevJuriaansPolcinoMilies97}. This state of affairs changed when Hertweck introduced a method to tackle the conjecture involving Brauer characters \cite{HertweckBrauer}. Nevertheless, results are still relatively far and between \cite{HertweckBrauer, HertweckA6, BovdiHertweck08, Gildea13, 4primaryII, KimmerleKonovalov17, Gitter}, and, for instance, the only non-abelian simple groups for which the conjecture has been verified are the groups $\operatorname{PSL}(2,q)$ where $q \leq 25$,  $q =32$ \cite{4primaryII} or where $q$ is a Fermat or Mersenne prime \cite{FermatMersenne}.

In the present article we show that the Zassenhaus conjecture is false by providing a series of metabelian groups $G$ such that $\mathbb{Z}G$ contains a unit of finite order not conjugate in $\mathbb{Q}G$ to any element of the form $\pm g$ for $g\in G$.

Let us describe these groups. To this end, let $p$ and $q$ be odd primes, $d$ an odd divisor of $p-1$ and $q-1$, $N$ the additive group $\mathbb{F}_{p^2} \oplus \mathbb{F}_{q^2}$, and let $\alpha$ and $\beta$ be primitive elements in the multiplicative groups $\mathbb{F}_{p^2}^\times $ and $\mathbb{F}_{q^2}^\times$, respectively. 
Consider the abelian group 
$$A = \langle a,b,c\ | \ a^{\frac{p^2-1}{d}}=b^{\frac{q^2-1}{d}}=1,\ c^d=a\cdot b \rangle$$ 
There is an action of $A$ on $N$ given by 
$$(x,y)^a = (\alpha^d\cdot x, y), \ (x,y)^b = (x, \beta^d\cdot y), (x,y)^c = (\alpha\cdot x, \beta\cdot y) $$
and we may form the semidirect product $N\rtimes A$, which we denote by $G(p,q;d;\alpha,\beta)$. The following are our main results:
% (see Definition \ref{defi Gpqd} for an explanation of this notation). 

\begin{namedtheorem}[Theorem A]
  Let $G=G(7,19;3;\alpha,\beta)$, where $\alpha$ is a root of the polynomial $X^2-X+3$ over $\mathbb{F}_7$ and $\beta$ is a root of $X^2-X+2$ over $\mathbb{F}_{19}$.
  There exists a unit $u \in \UU(\mathbb{Z}G)$ of order $7 \cdot 19$ such that $u$ is not conjugate in $\mathbb{Q}G$ to any element of the form $\pm g$ for $g \in G$. In particular, the Zassenhaus conjecture does not hold for $G$.
\end{namedtheorem}

\begin{namedtheorem}[Theorem B]
 Let $d$ be an odd positive integer, and let $N\in \N$ be arbitrary. There exist infinitely many pairs of primes $p$ and $q$ such that, for any admissible choice of $\alpha$ and $\beta$, for $G=G(p,q;d;\alpha,\beta)$ there are $u_1,\ldots,u_N \in \UU(\mathbb{Z}G)$, each of order $p \cdot q$, such that neither one of the $u_i$ is conjugate in $\UU(\mathbb{Q}G)$ to an element of the form $\pm g$ for $g \in G$, or to any other $u_j$ for $j\neq i$. In particular, the Zassenhaus conjecture does not hold for such a group $G$.
\end{namedtheorem}
A more precise version of Theorem B, specifying lower bounds for $p$ and $q$ as well as the rational conjugacy classes of the $u_i$, can be found in Corollary~\ref{corollary infinite series}.
The idea that groups like $G(p,q;d;\alpha,\beta)$ might be good candidates for a counterexample to the Zassenhaus conjecture was noted in \cite{MargolisdelRioCW3}. Looking at the various positive results mentioned above, it seems that metabelian groups would have been the next logical step, and people working in the field certainly attempted to prove the Zassenhaus conjecture for metabelian groups, to no avail. 
%In this light, it does not seem too surprising that there should be a metabelian counterexample. 
What is more, the class of metabelian groups provided E.~Dade's counterexample to R.~Brauer's question, which asked whether $KG \cong KH$ for all fields $K$ implies that $G$ and $H$ are isomorphic \cite{Dade}. The second Zassenhaus conjecture mentioned above, which asked if different (normalised) group bases of $\mathbb{Z}G$ are conjugate in $\mathbb{Q}G$, fails for metabelian groups as well \cite{Klingler}. On the other hand, metabelian groups were one of the first classes of groups for which the isomorphism problem on integral group rings was known to have a positive answer \cite{Whitcomb}.

Here is an outline of our strategy to prove Theorems A and B:
\begin{enumerate}
\item If $U$ is a cyclic group of order $n$, then a unit $u\in \UU(RG)$ of order $n$ corresponds to a certain $R(G\times U)$-module $_u(RG)_G$, called a ``double action module''. This is the well-known double action formalism explained in Section \ref{section double action}, and the defining property of double action modules is that their restriction to $G$ is a free $RG$-module of rank one.
This principle works for any commutative ring $R$. 
\item Once we fix a conjugacy class of units of order $n$ in $\UU (\Q G)$, or equivalently a $\Q (G\times U)$-double action module $V={_u(\Q G)_G}$ corresponding to it, we need to find a $\Z(G\times U)$-lattice in $V$ whose restriction to $G$ is free. 
\item Let $\Z_{(p)}$ denote the localisation of $\Z$ at the prime ideal $(p)$. We provide a fairly general construction of double action modules over $\Z_{(p)}(G\times U)$ for groups of the form $N\rtimes A$, where $N$ is abelian. This is done in Section \ref{section semilical counterexample}, and, of course, subject to a whole list of conditions. The double action modules we construct are direct sums of direct summands of permutation modules (see Definition \ref{defi MXpq}), and as a consequence the local version of the counterexample is fairly explicit (see Proposition \ref{prop explicit comp} at the end). 
\item The problem of turning a family of ``compatible'' $\Z_{(p)}(G\times U)$-lattices in $V$ with free restriction to $G$ into 
a $\Z (G\times U)$-lattice in $V$ with the same property can be solved using a rather general local-global principle, provided the centraliser $C_{\UU(\Q G)}(u)$ of 
the unit is big enough (think of $u$ as already being fixed up to conjugacy in $\UU(\Q G)$).  This is done in Section \ref{section existence global}.
\item In the last section we study groups of the form $G(p,q;d;\alpha,\beta)$ as defined above. All of the more general results of the preceding sections become explicit and elementarily verifiable in this situation.
We use the general result of that section, Theorem~\ref{thm zc Gpqd}, to prove Theorems A and B.
%show that $G(7,19;3;\alpha,\beta)$ for our choice of $\alpha$ and $\beta$ is a counterexmple to the Zassenhaus conjecture, but the same can be done for many other groups of the form $G(p,q;d;\alpha,\beta)$.
\end{enumerate}

In regard to future research, it seems worth pointing out that many variations and weaker versions of the Zassenhaus conjecture remain open. An overview of the weaker forms of the conjecture can be found in \cite{MargolisdelRioPAP}. In particular, the question if the orders of torsion units of augmentation one in $\UU(\mathbb{Z}G)$ coincide with the orders of elements in $G$ remains open. It also might still be true that if $u \in \UU(\mathbb{Z}G)$ is a torsion unit then $u$ is conjugate in $\UU(\Q H)$ to $\pm g$ for some $g\in G$, where $H\supseteq G$ is some larger group containing $G$.

Going in a different direction,  the $p$-version of even the strongest of the three Zassenhaus conjectures remains open. This variation asks if it is true that any $p$-subgroup of $\UU(\mathbb{Z}G)$ consisting of elements of augmentation one is conjugate in $\UU(\mathbb{Q}G)$ to a subgroup of $G$. This is sometimes called ``(p-ZC3)'' or the ``Strong Sylow Theorem'' for $\mathbb{Z}G$. An overview of results relating to this problem can be found in \cite{BaechleKimmerleMargolisDFG}. For the counterexample to the Zassenhaus conjecture presented in the present article it is of fundamental importance that the order of the unit is divisible by at least two different primes.

Throughout the paper we are going to use the following notation, most of which is quite standard. 
\begin{namedtheorem}[Notation and basic definitions]\normalfont
	\begin{enumerate}
	\item Let $G$ be a finite group and let $U$ be a subgroup of $G$. For a character $\chi$ of $U$ we write $\chi\uparrow^G_U$ for the induced character, and for a character $\psi$ of $G$ we write $\psi|_U$ for the restriction to $U$. The trivial character of $G$ is denoted by $1_G$. 
	
	\item For a prime number $p$ we denote by $G_p$ a Sylow $p$-subgroup of $G$ and by $g_p$ the $p$-part of an element $g \in G$. The conjugacy class of $g \in G$ is denoted by $g^G$. We also use $G_{p'}$ to denote a $p'$-Hall subgroup of $G$ (this is used only for nilpotent $G$) and $g_{p'}$ for the $p'$-part of $g$.
	
	\item Let $G$ be a finite group and let $H_1$ and $H_2$ be subgroups of $G$ such that $G=H_1\times H_2$.
	If $\chi_1$ and $\chi_2$ are characters of $H_1$ and $H_2$, respectively, then $\chi_1\otimes \chi_2$ denotes the corresponding character of $G$. Similarly, if $L_1$ and $L_2$ are $RG$-modules for some commutative ring $R$,
	$L_1\otimes L_2$ denotes $L_1\otimes_R L_2$ with the natural $RG$-module structure.
	
	\item We write ``$\sum_{g^G}$'' to denote a sum ranging over a set of representatives of the conjugacy classes of $G$. If $G$ acts on a set $H$ we write ``$\sum_{h^G, h \in H}$'' for the sum ranging over representatives of the $G$-orbits in $H$.
	\item If $G$ and $H$ are groups, and $H$ acts on $G$ by automorphisms,  we denote by $\Irr_{\Q}(G)/H$ the set of $H$-orbits of irreducible rational characters of $G$. We write ``$\sum_{\varphi \in \Irr_{\Q}(G)/H}$'' for a sum ranging over representatives of these orbits.
	\item For a cyclic group $U = \langle c \rangle$ and an element $g \in G$ we define
	\begin{equation}
	[g] = \langle (g, c) \rangle \leq G\times U
	\end{equation}
	We will often use the fact that $[g]_p=\langle (g_p, c_p) \rangle$ and 
	$[g]_{p'}=\langle (g_{p'}, c_{p'}) \rangle$ for all primes $p$.
	
	\item  If $R$ is a ring and $M$ is an $R$-module we write $M^{\oplus n}$ for the direct sum of $n$ copies of $M$.
	\item If $R$ is a ring, $M$ is an $R$-module and $X\subseteq M$ is an arbitrary subset, we write $R\cdot X$ for the $R$-module generated by $X$.
	\item Let $R$ be a commutative ring and let 
	$u =\sum_{g \in G}r_g\cdot  g$
	be an arbitrary element of $RG$. 
	Then 
	$$\varepsilon_{g^G}(u) = \sum_{h \in g^G} r_h$$
	is called the \emph{partial augmentation} of $u$ at $g$.
	\end{enumerate}
\end{namedtheorem}

\section{Double action formalism}\label{section double action}

The ``double action formalism'' (see, for instance, \cite[Section 38.6]{SehgalBook}) is a commonly used way of studying the Zassenhaus conjecture and other questions relating to units in group algebras via certain bimodules, the so-called ``double action modules''. In this section we give a short (but complete, at least for our purposes) overview of this formalism. For the rest of this section 
let $G$ be a finite group and let $U=\langle c \rangle$ be a cyclic group of order $n\in \N$. By $R$ we denote an arbitrary commutative ring.

\begin{defi}
	\begin{enumerate}
		\item Given a unit $u\in \UU(RG)$ satisfying $u^n=1$ we define an $R (G\times U)$-module ${_u}(RG)_G$ as follows: as an $R$-module, ${_u}(RG)_G$ is equal to $RG$, and the (right) action of $G\times U$ is given by
		\begin{equation}
			{_u (RG)_G}\times (G\times U) \longrightarrow {_u(RG)_G}:\ (x, (g, c^i)) \mapsto (u^\circ)^i\cdot x \cdot g
		\end{equation} 
		where the product on the right hand side of the assignment is taken within the ring $RG$, and $-^\circ:\ RG \longrightarrow RG: g \mapsto g^{-1}$ denotes the standard involution on $RG$.
		We call this $R(G\times U)$-module the \emph{double action module}
		associated with the unit $u$.
		\item An $R(G\times U)$-module $M$ is called \emph{$G$-regular} if 
		$M|_{G}$ is free of rank one as an $RG$-module (that is, it is isomorphic to $RG$ considered as a right module over itself).
	\end{enumerate}
\end{defi}

A double action module is clearly $G$-regular, but it turns out that the converse is true as well:
\begin{prop}\label{prop double action iso implies conj}
	\begin{enumerate}
		\item If $M$ is a $G$-regular $R(G\times U)$-module, then $M\cong {_u(RG)_G}$ for some unit 
		$u\in \UU(RG)$ with $u^n=1$.
		\item If $u,v\in\UU(RG)$ are two units satisfying $u^n=v^n=1$, then
		$${_u(RG)_G}\cong{_v(RG)_G}$$
		if and only if $u$ and $v$ are conjugate inside $\UU(RG)$.
	\end{enumerate}
\end{prop}
\begin{proof}
	Assume that $M$ is $G$-regular. Then we may choose an isomorphism of $RG$-modules
	$\varphi:\ M \longrightarrow RG$, where we view $RG$ as a right module over itself.
	As before, let $-^\circ:\ RG \longrightarrow RG:\ g \mapsto g^{-1}$ denote the standard involution of the group algebra, and define $u=\varphi(\varphi^{-1}(1)\cdot c)^\circ$. Then we have, for all $m\in M$,
	$$
		\varphi(m\cdot c) = \varphi(\varphi^{-1}(1)\cdot \varphi(m)\cdot c) = \varphi(\varphi^{-1}(1)\cdot c)\cdot \varphi(m)=u^\circ\cdot \varphi(m)
	$$ 	
	where we made use of the fact that $m=\varphi^{-1}(1)\cdot \varphi(m)$. It now follows immediately from the above that $u^n=1$, and the map $M\longrightarrow {_u(RG)_G}:\ x \mapsto \varphi(x)$ is easily seen to be an isomorphism of $R(G\times U)$-modules.
	
	Let us now prove the second part of the proposition. To this end, fix an isomorphism 
	$\varphi:\ {_u(RG)_G} \longrightarrow {_v(RG)_G}$. Then $\varphi(1)\cdot u^\circ =\varphi(1\cdot u^\circ)=\varphi(u^\circ) = \varphi(1\cdot c) = \varphi(1)\cdot c=v^\circ \cdot \varphi(1)$. As an equation purely in the ring $RG$ this yields $u\cdot \varphi(1)^\circ=\varphi(1)^\circ\cdot v$. So it only remains to show that $\varphi(1)^\circ$ is an invertible element of $RG$, which follows from the fact that $\varphi(1)$ generates ${_v(RG)_G}$ as an $RG$-module.  
\end{proof}

As we have seen so far, double action modules are in one-to-one correspondence with conjugacy classes of elements of $\UU (RG)$ whose order divides $n$. Evidently this means that each property of torsion units should have a counterpart in the language of double action modules. 
An important tool in the study of the Zassenhaus conjecture is the criterion  given in \cite[Theorem 2.5]{MRSW87}. It states that a unit $u \in \UU(\mathbb{Z}G)$ of finite order and augmentation one is conjugate in $\UU(\mathbb{Q}G)$ to an element of $G$ if and only if $\varepsilon_{g^G}(u^i) \geq 0$
for all $g\in G$ and all $i\geq 0$.
In particular, finding a counterexample to the Zassenhaus conjecture is equivalent to finding a unit $u$ of finite order and augmentation one which has a negative partial augmentation.
Hence it is important for us to have a way of recovering the partial augmentations of a unit from the corresponding double action module.

\begin{prop}\label{prop character double action module}
	Let $u\in \UU(RG)$ be a unit satisfying $u^n=1$. Let 
	\begin{equation}
	\theta_u:\ G\times U \longrightarrow R
	\end{equation}
	denote the character of the $R(G\times U)$-module $_u(RG)_{G}$. Then 
	\begin{equation}
	\theta_u((g,c^i))=|C_G(g)| \cdot \eps_{g^G}(u^i)
	\end{equation}
\end{prop}
\begin{proof}
	Let us first calculate the trace of the linear map $\mu(g,h):\ RG\longrightarrow RG:\ x \mapsto g^{-1}\cdot x\cdot h$ for arbitrary $g,h\in G$.
	This trace is equal to the number of $y\in G$ such that $g^{-1}\cdot y\cdot h=y$, or, equivalently, 
	$h=y^{-1}\cdot g \cdot y$.  If $g^G=h^G$ then this number is equal to $|C_G(g)|$, otherwise it is zero.
	Now, if
	$$
		u^i=\sum_{g\in G} \alpha_g\cdot g
	$$ 
	then the linear endomorphism of $RG$ induced by $(g, c^i)$ is equal to $\sum_{h\in G} \alpha_h\cdot \mu(h,g)$.
	The character value $\theta_u((g,c^i))$ is the trace of this map, which is equal to
	$$
		\sum_{h\in G} \alpha_h\cdot \Tr(\mu(h,g)) = \sum_{h\in g^G} \alpha_h\cdot |C_G(g)| = |C_G(g)|\cdot \eps_{g^G}(u^i)	
	$$
	as claimed.
\end{proof}

We now turn our attention to the case of rational coefficients, i.e. $R=\Q$. In that situation $G$-regularity of $G\times U$-modules can readily be checked on the level of characters, and Proposition~\ref{prop character double action module} can be used to ascertain whether the corresponding torsion unit in $\UU (\Q G)$ is indeed not conjugate to an element of $G$. 
%But first let us fix a piece of notation which we will use throughout this article:
%\begin{defi}
%	For $g\in G$ we define
%	\begin{equation}
%		[g] = \langle (g, c) \rangle \leq G\times U
%	\end{equation}
%	We will often use the fact that $[g]_p=\langle (g_p, c_p) \rangle$ and 
%	$[g]_{p'}=\langle (g_{p'}, c_{p'}) \rangle$ for all primes $p$.
%\end{defi}

\begin{prop}\label{prop sum induced partaug}
	Let $g_1,\ldots, g_k$ be pair-wise non-conjugate elements of $G$ whose order divides $n$, and let $a_1,\ldots,a_k\in \Z$ such that $a_1+\ldots+a_k=1$. Assume that
	\begin{equation}\label{eqn theta sum ai}
		\theta = \sum_{i=1}^k a_i \cdot 1\uparrow_{[g_i]}^{G\times U}
	\end{equation}
	is in fact a character of $G\times U$, rather than just a virtual character. Then $\theta$ is the character of $_u{(\Q G)}_G$
	for some $u\in \UU(\Q G)$ satisfying $u^n=1$. Moreover $\eps_{g_i^G}(u)=a_i$ for all $i\in\{1,\ldots,k\}$ and $\eps_{g^G}(u)=0$ whenever $g$ is not conjugate to any of the $g_i$. 
\end{prop}
\begin{proof}
	Let us first prove that $\theta$ can be realised as the character of a $\Q (G\times U)$-module, rather than just a $\C (G\times U)$-module. 
	By definition $\theta$ can be written as the difference of the characters of two $\Q (G\times U)$-modules, say $V$ and $W$. Without loss of generality we may assume that $V$ and $W$ share no isomorphic simple direct summands. But then $\Hom_{\Q (G\times U)}(V,W)=0$, which implies $\Hom_{\C (G\times U)}(\C\otimes_\Q V,\C\otimes_\Q W)=0$. That is, $\C\otimes_\Q V$ and $\C \otimes_\Q W$ share no isomorphic simple direct summands, which means that $\theta$ can only be a proper character if $W=\{0\}$, which means that $V$ is a $\Q (G\times U)$-module affording $\theta$.
	 
	To verify that $\theta$ is the character of  $_u{(\Q G)}_G$
	for some $u\in \UU(\Q G)$ satisfying $u^n=1$, it suffices to show that $\theta|_G$ is equal to the regular character of $G$, which is equal to $1\uparrow_{\{ 1\}}^G$. Note that by Mackey's theorem we have
	$$
		1\uparrow_{[g_i]}^{G\times U}|_G = \sum_{x} 1\uparrow_{[g_i]^x \cap G}^{G} 
	$$
	where $x$ ranges over a transversal for the double cosets $[g_i]\setminus G\times U/G$. Since $[g_i]\cdot G = G\times U$, there is just one such double coset, and therefore
	$$
		1\uparrow_{[g_i]}^{G\times U}|_G = 1\uparrow_{[g_i] \cap G}^{G}
		=1\uparrow_{\{ 1\}}^G 
	$$
	independent of $i$. Combining this fact with \eqref{eqn theta sum ai} we get
	$$
		\theta|_G=\sum_{i=1}^k a_i\cdot 1\uparrow_{\{ 1 \}}^G	=1\uparrow_{\{ 1 \}}^G
	$$
	All that is left to prove now is our claim on the partial augmentations. We know by now that $\theta=\theta_u$ for some $u$, with $\theta_u$ as defined in Proposition~\ref{prop character double action module}. Thus, Proposition~\ref{prop character double action module}  yields  
	\begin{equation}\label{eqn formula eps pa}
		\eps_{g^G}(u)=\frac{\theta((g,c))}{|C_G(g)|}
		=\frac{1}{|C_G(g)|}\sum_{i=1}^k a_i \cdot 1\uparrow_{[g_i]}^{G\times U}((g,c))
	\end{equation}
	The character $1\uparrow_{[g_i]}^{G\times U}$ evaluated on $(g,c)$ is equal to the number of $h\in G$ such that $(g,c)\in [g_i^h]=\langle(g_i^h,c)\rangle$. Since the order of $g_i$ divides $n$, which is the order of $c$, it follows that the projection from $[g_i^h]$ to $U$ is an isomorphism, and hence $(g_i^h,c)$ is the only element of $[g_i^h]$ whose projection to $U$ is $c$. This implies that $(g,c)\in[g_i^h]$ if and only if $g=g_i^h$. It follows that $1\uparrow_{[g_i]}^{G\times U}((g,c))$ is equal to zero if $g^G \neq g_i^G$, and equal to $|C_G(g)|$ if $g^G=g_i^G$. Plugging this back into \eqref{eqn formula eps pa} yields the desired result for the partial augmentations of $u$.
\end{proof}

\section{Local and semi-local rings of coefficients}

Let $R$ be the ring of integers in an algebraic number field $K$.
For a maximal ideal $p$ of $R$ we let $R_{(p)}$ denote the localisation of $R$ at the prime $p$. If
$\pi=\{p_1,\ldots,p_k\}$ is a finite collection of maximal ideals of $R$, we define
$$
	R_\pi = \bigcap_{p\in \pi} R_{(p)},
$$
which is a semi-local ring whose maximal ideals are precisely $p_i\cdot R_\pi$ for $i=1,\ldots,k$.
Testing whether a particular module is a double action module of a unit is particularly easy over 
$R_{(p)}$, as the following proposition shows:
\begin{prop}\label{prop identify local double action}
	Let $M$ be an $R_{(p)}(G\times U)$-module such that $M|_G$ is projective and $K\otimes_{R_{(p)}} M$ is $G$-regular. Then $M$ is $G$-regular.
\end{prop}
\begin{proof}
	By assumption $M|_G$ is projective and its character is equal to the character of the regular $R_{(p)}G$-module. 
	We need to show that this implies that $M|_G$ is isomorphic to the regular $R_{(p)}G$-module. This follows from the fact that two projective $R_{(p)}G$-modules are isomorphic if and only if their characters are the same, a consequence of the fact that the decomposition matrix of a finite group has full row rank (see \cite[Corollary 18.16]{CurtisReinerI} for the precise statement we are using).
\end{proof}

Constructing a $G$-regular $R_\pi (G\times U)$-module is actually equivalent to constructing a $G$-regular
$R_{(p)}(G\times U)$-module for each $p\in \pi$ in such a way that all of these modules have the same character. 

\begin{prop}\label{prop local to semilocal}
	Let $\Lambda$ be an $R$-order in a finite-dimensional semisimple $K$-algebra $A$ and let $V$ be a finite-dimensional $A$-module.
	\begin{enumerate}
		\item Assume that we are given full $R_{(p)}\Lambda$-lattices $L(p)\leq V$ for each $p\in\pi$. Then
		$$
			L=\bigcap_{p\in \pi} L(p)
		$$
		is an $R_\pi\Lambda$-lattice in $V$ with the property that 
		$R_{(p)}L=L(p)$ for each $p\in \pi$.
		\item Given two $R_\pi \Lambda$-lattices $L$ and $L'$ in $V$, we have $L\cong L'$ if and only if $R_{(p)}L\cong R_{(p)}L'$ for each $p\in \pi$. 
	\end{enumerate}
\end{prop} 
\begin{proof}
	$L$ is clearly a $R_\pi\Lambda$-module, and in order to show that it is a lattice it suffices to show that it is contained in some $R_\pi$-lattice. If $L'$ is an arbitrary full $R$-lattice in $V$, then for each $p\in \pi$ there is a number $e(p)\in \Z_{\geq 0}$ such that $L(p)\subseteq p^{-e(p)}\cdot R_{(p)}L'$. Let $0\neq N\in \Z$ be a number such that $p^{e(p)} \supseteq N\cdot R$ for each $p\in \pi$. Then $L \subseteq N^{-1}\cdot R_{(p)}L'$ for each $p\in \pi$, and therefore $L\subseteq \bigcap_{p\in \pi} N^{-1}\cdot R_{(p)}L' = N^{-1}\cdot R_\pi L'$, which is an $R_{\pi}$-lattice.
 	
	Now let us prove that $R_{(p)}L=L(p)$ for each $p\in\pi$. Clearly $R_{(p)}L \subseteq L(p)$. On the other hand, 
	if $v\in L(p)$, then there is an integer $N$ such that $N\cdot v \in 
	L(q)$, for all $q \in \pi$ with $q \neq p$, and $N \not\equiv 0 \ ({\rm mod }\ p)$
	(we can take $N$ to be contained in a product of sufficiently large powers of the maximal ideals in $\pi$ different from $p$). By definition, we now have 
	$N\cdot v \in L$, and since $N$ is invertible in $R_{(p)}$ we also have $v=N^{-1}\cdot N\cdot v \in R_{(p)}L$. This implies 
	$L(p)\subseteq R_{(p)}L$, which completes the proof of the first point.
	For the second point see \cite[Exercise 18.3]{Reiner}.
\end{proof}

\section{Locally free lattices and class groups}

As in the previous section let $R$ be the ring of integers in an algebraic number field $K$.
Let $A$ be a finite-dimensional semisimple $K$-algebra, and let $\Lambda$ be an $R$-order in $A$. 
Throughout this section we adopt the following notational convention: if $p$ is a maximal ideal of $R$, and $M$ is an $R$-module, then $M_p$ denotes the $p$-adic completion of $M$. In particular, $K_p$ is a complete field with valuation ring $R_p$, $A_p$ is a finite-dimensional $K_p$-algebra and $\Lambda_p$ is an $R_p$-order in $A_p$. 

Let us first check that no information is lost in passing from the localisations considered in the previous section to the completions we are going to consider now.
If we keep the notation $R_{(p)}$ for the localisation of $R$ at $p$ and $\Lambda_{(p)}=R_{(p)}\cdot \Lambda \subseteq A$, then 
$R_p$ and $\Lambda_p$ can also be viewed as the $p$-adic completions of $R_{(p)}$ and $\Lambda_{(p)}$, respectively. 
\begin{prop}[{\cite[Proposition 30.17]{CurtisReinerII}}]
	Let $M$ and $N$ be finitely generated $\Lambda_{(p)}$-modules. Then
	$$
		M\cong N \quad\textrm{if and only if}\quad M_p \cong N_p
	$$
\end{prop}
In particular, if $M$ and $N$ are finitely generated $\Lambda$-modules, then $M_{(p)}\cong N_{(p)}$ as $\Lambda_{(p)}$-modules if and only if $M_p\cong N_p$ as $\Lambda_p$-modules.

Now let us define the protagonist of this section: the locally free class group of $\Lambda$. 
\begin{defi}[{cf. \cite[\S 49A]{CurtisReinerII}}]
	\begin{enumerate}
	\item A right $\Lambda$-lattice $L$ is called \emph{locally free} of rank $n\in \N$ if
	$$
		L_p \cong \Lambda_p^{\oplus n}
	$$
	as right $\Lambda_p$-modules for all maximal ideals $p$ of $R$.
	\item If $L$ and $L'$ are right $\Lambda$-lattices, we say that $L$ and 
	$L'$ are \emph{stably isomorphic} if 
	$$
		L\oplus \Lambda^{\oplus n} \cong L'\oplus \Lambda^{\oplus n}
	$$ 
	for some $n\in \N$.
	\item The \emph{locally free class group} of $\Lambda$, denoted by $\Cl(\Lambda)$,
	 is an additive group whose elements are the stable isomorphism classes $[X]$ of locally free right $\Lambda$-ideals in $A$.
	The group operation on $\Cl(\Lambda)$ is defined as follows:
	if $X$ and $Y$ are locally free right $\Lambda$-ideals in $A$, then 
	there is a locally free right $\Lambda$-ideal $Z$ such that
	$$
		X\oplus Y \cong Z \oplus \Lambda
	$$
	as right $\Lambda$-modules. We define the sum $[X]+[Y]$ to be equal to $[Z]$. 
	\end{enumerate}
\end{defi}

Note that the unit element of $\Cl(\Lambda)$ is $[\Lambda]$.
For the purposes of this article, class groups serve as a means to prove that certain 
$\Lambda$-lattices are free. The reason this works is that most group algebras satisfy the Eichler condition relative to $\Z$, which guarantees that we can infer $X\cong \Lambda$ from $[X]=[\Lambda]$:
\begin{defi}[{cf. \cite[Remark 45.5 (i)]{CurtisReinerII}}]
	We say that $A$ satisfies the \emph{Eichler condition} relative to $R$ if no
	simple component of $A$ is isomorphic to a totally definite quaternion algebra.
\end{defi}

\begin{thm}[Jacobinski Cancellation Theorem {\cite[Theorem 51.24]{CurtisReinerII}}]\label{thm jacobinski}
	If $A$ satisfies the Eichler condition relative to $R$, then any two locally free $\Lambda$-lattices which are stably isomorphic are isomorphic.
\end{thm}

\begin{thm}[{\cite[Theorem 51.3]{CurtisReinerII}}]\label{thm qg eichler}
	If $G$ is a finite group which does not have an epimorphic image isomorphic to either 
	one of the following: 
	\begin{enumerate}
		\item A generalised quaternion group of order $4n$ where $n \geq 2$.
		\item The binary tetrahedral group of order $24$.
		\item The binary octahedral group of order $48$.
		\item The binary icosahedral group of order $120$.
	\end{enumerate}
	then $KG$ satisfies the Eichler condition relative to $R$.
\end{thm}

We now turn our attention to the problem of deciding whether a given locally free $\Lambda$-ideal is trivial in $\Cl(\Lambda)$.

\begin{defi}[{\cite[(49.4)]{CurtisReinerII}}]
	\begin{enumerate}
	\item We define the id\`ele group of $A$ as
	$$
		J(A) = \left\{ (\alpha_p)_p \in \prod_{p} \UU(A_p) \ \Bigg| \  \alpha_p \in \UU(\Lambda_p) \textrm{ for all except finitely many $p$} \right\}
	$$
	where $p$ ranges over all maximal ideals of $R$. If $\alpha=(\alpha_p)_p$ and $\beta=(\beta_p)_p$ are two elements of $J(A)$, then their product $\alpha\cdot \beta$ in $J(A)$ is defined as $(\alpha_p\cdot \beta_p)_p$.
	\item We identify $\UU(A)$ with the subgroup of $J(A)$ consisting of constant id\`eles. 
	\item Define
	$$
		U(\Lambda) = \{ \alpha \in J(A) \ | \ \alpha_p \in \UU(\Lambda_p) \textrm{ for all $p$}  \}
	$$
	This is also a subgroup of $J(A)$.
	\end{enumerate}
\end{defi}
Even though it is not immediately obvious from the definition, $J(A)$ does not depend on the order $\Lambda$ (in fact, if $\Gamma$ is another $R$-order in $A$, then $\Lambda_p=\Gamma_p$ for all except finitely many $p$).

\begin{thm}[{Special case of \cite[Theorem 31.18]{CurtisReinerII}}]
	There is a bijection between the double cosets
	$$
		\UU(A) \setminus J(A) / U(\Lambda)
	$$
	and isomorphism classes of locally free right $\Lambda$-ideals in $A$ given by
	$$
		\UU(A) \cdot \alpha \cdot U(\Lambda) \mapsto A \cap \bigcap_{p} \alpha_p\Lambda_p	
	$$
	where $\alpha=(\alpha_p)_p\in J(A)$. We denote the right hand side of this assignment by $\alpha\Lambda$.
\end{thm}

As shown in \cite[Theorem 31.19]{CurtisReinerII} we have, for arbitrary 
$\alpha, \beta \in J(A)$, an isomorphism $\alpha\Lambda \oplus \beta\Lambda \cong \Lambda \oplus \alpha\beta\Lambda$. This shows that there is an epimorphism of groups
$$
	\theta:\ J(A) \twoheadrightarrow \Cl(\Lambda):\ \alpha \mapsto [\alpha\Lambda]
$$
Since $\Cl(\Lambda)$ is commutative by definition, we certainly have 
$[J(A), J(A)]\subseteq \Ker(\theta)$. In \cite{FroehlichClassGroups}, A.~Fr\"ohlich gave an explicit characterisation of the kernel of $\theta$, which will be very useful to us later. 
\begin{defi}[Reduced norms]
	Let $F$ be a field, and let $B$ be a finite-dimensional semisimple
	$F$-algebra. Then there is a decomposition
	$$
		B = B_1\oplus \ldots \oplus B_n
	$$
	where each $B_i$ is a simple $F$-algebra. We may view $B_i$ as a central simple algebra over its centre $Z(B_i)$. In each component we have a reduced norm map
	$$
		\nr_{B_i/Z(B_i)}:\ B_i \longrightarrow Z(B_i)
	$$
	obtained by embedding $B_i$ into $E\otimes_{Z(B_i)} B_i$ for some field extension $E/Z(B_i)$ which splits $B_i$,
	followed by mapping $E\otimes_{Z(B_i)} B_i$ isomorphically onto a full matrix ring over $E$ of the appropriate dimension and then taking the determinant.
	We can then define a reduced norm map on $B$ component-wise. This map will take values in
	$Z(B_1)\oplus \ldots \oplus Z(B_n)=Z(B)$. That is, we get a multiplicative map 
	$$
		\nr_{B/Z(B)}:\ B \longrightarrow Z(B)
	$$
\end{defi}

\begin{defi}
	Define
	$$
		J_0(A) = \{ (\alpha_p)_p\in J(A) \ | \ \nr_{A_p/Z(A_p)}(\alpha_p)=1 \textrm{ for all $p$} \}
	$$
	This is a normal subgroup of $J(A)$.
\end{defi}

\begin{thm}[{\cite[Theorem 1 and subsequent remarks]{FroehlichClassGroups}}]\label{thm froehlich}
	The map
	$$
		\frac{J(A)}{J_0(A)\cdot \UU(A) \cdot U(\Lambda)} \stackrel{\sim}{\longrightarrow} \Cl(\Lambda):\ 
		\alpha\cdot {J_0(A)\cdot \UU(A) \cdot U(\Lambda)} \mapsto [\alpha\Lambda]
	$$
	is an isomorphism of groups.
\end{thm}
For our purposes it will suffice to know that for any $\alpha\in J_0(A)$ the corresponding element $[\alpha\Lambda]\in \Cl(\Lambda)$ is trivial.

\section{Semi-local counterexamples}\label{section semilical counterexample}
After these general sections we will now start to work with a more concrete class of groups which will ultimately provide our counterexample.
Let $G$ be a finite group of the form
\begin{equation}
	G = N \rtimes A
\end{equation}
where $N$ is an abelian group. Moreover, let $U=\langle c \rangle$ be a cyclic group such that the exponent of $N$ and the exponent of $U$ coincide. Let
$$
	\eps\OfU:\ G \longrightarrow \Z:\ g \mapsto \eps_{g^G}\OfU
$$
be a class function which vanishes outside of $N$ (the notation for $\eps$ is deliberately chosen to resemble our notation for partial augmentations). When we say that the partial augmentations of a unit $u \in \UU(RG)$ are given by $\varepsilon$, for some commutative ring $R\supseteq \Z $, we mean that $\varepsilon_{g^G}(u) =  \varepsilon_{g^G}$ for all $g \in G$. Define
\begin{equation}\label{eqn def chi}
	\chi = \sum_{g^G} \eps_{g^G}\OfU\cdot 1\uparrow_{[g]}^{G\times U} 
\end{equation}
%where $g$ runs over a set of representatives for the conjugacy classes of $G$. 
Note that, a priori, $\chi$
is only a virtual character of $G\times U$.
Assume that all of the following hold:
\begin{enumerate}[label=\textbf{(C.\arabic*})]
%\item $\chi$ is a proper character of $G\times U$ (i.e., not just a virtual character).
\item{\label{cond sum 1}} $\sum_{g^G} \eps_{g^G}\OfU=1$.
%, where $g$ runs over representatives of the conjugacy classes of $G$.
%\item If there are $m,n\in N$ of the same order as $c$ such that both $\eps_{m^G}\OfU\neq 0$ and $\eps_{n^G}\OfU \neq 0$, then $m_p^G=n_p^G$ for each prime $p$ dividing the order of $N$. 
\item{\label{cond cent n eq N}}If $\eps_{n^G}\OfU\neq 0$ for some $n\in N$, then 
$C_G(n_p)\cap C_G(n_{p'}^g)=N$ for all $g\in G$ and all primes $p$ dividing the order of $N$.
\item\label{cond xi proper char} For each prime $p$ dividing the order of $N$ we have a decomposition
\begin{equation}\label{eqn decomp xin}
	\chi|_{N\times U} = \sum_{n \in N_p} \xi_n\otimes 1\uparrow_{[n]_p}^{N_p\times U_p}
\end{equation} 
where $\xi_n$ is a proper character of $N_{p'}\times U_{p'}$ for each
$n\in N_p$.
\end{enumerate}

The aim of this section is to prove the following theorem:
\begin{thm}\label{thm existence u semilocal}
	Let $\pi$ be a finite collection of primes. Then,
	under the above assumptions, there exists a $G$-regular $\Z_\pi (G\times U)$-lattice $L$ with
	character $\chi$. Moreover, the partial augmentations of the associated unit $u_\pi \in \UU(\Z_\pi G)$ are given by $\eps\OfU$.
\end{thm}

By \cite[Theorem 3.3]{CliffWeiss} the condition \ref{cond xi proper char} actually implies that there exists a $\mathbb{Z}(N \times U)$-lattice which is locally free over $\mathbb{Z}N$.  We will not use this fact, but it provided the original motivation for this construction. The condition is also studied in \cite{MargolisdelRioCW3}. 

\begin{remark}\label{remark formula xin}
	In \eqref{eqn decomp xin}, the $\xi_n$ are uniquely determined. Namely,
	$$
		\xi_n = \sum_{m \in N_{p'}} \varepsilon_{(m\cdot n)^G}\OfU\cdot 1\uparrow_{[m]_{p'}}^{N_{p'}\times U_{p'}}
	$$
\end{remark}
\begin{proof}
	By Mackey decomposition we have
	$$
	1\uparrow_{[h]}^{G \times U}|_{N\times U} = [C_G(h):N] \sum_{m \in h^G} 1\uparrow_{[m]}^{N \times U}
	$$
	for every $h \in N$.
	Thus we obtain 
	$$
		\chi|_{N \times U} = \sum_{h \in N}[C_G(h):N] \cdot \varepsilon_{h^G}\OfU\cdot 1\uparrow_{[h]}^{N \times U}
	$$
	Note that our assumptions imply that $C_G(h)=C_G(h_p)\cap C_G(h_{p'})$ is equal to $N$ whenever $\eps_{h^G}\OfU\neq 0$. Hence
	$$
		\chi|_{N\times U}=\sum_{h\in N} \eps_{h^G}\OfU\cdot 1 \uparrow_{[h]}^{N\times U} 
	$$
	So, setting
	$$
		\xi_{n} = \sum_{m \in N_{p'}} \varepsilon_{(m\cdot n)^G}\OfU \cdot 1\uparrow_{[m]_{p'}}^{N_{p'}\times U_{p'}} 
	$$
	for every $n\in N_p$ certainly ensures that \eqref{eqn decomp xin} holds. 
	
	To prove that the $\xi_n$'s are uniquely determined as virtual characters, it suffices to show that they can be recovered from $\chi|_{N\times U}$. Let $m,n\in N_p$. Then $$
		1\uparrow_{[m]_p}^{N_p\times U_p}((n,c_p))= \left\{\begin{array}{ll}
		0 & \textrm{if $n \neq m$} \\ 
		|N_p| & \textrm{otherwise}
		\end{array}\right.
	$$ 
	Hence, if $h\in N_{p'}$ and $n\in N_p$, then
	$$
		\chi((h\cdot n, c)) = |N_p|\cdot \xi_n((h, c_{p'}))
	$$
	which shows that $\xi_n$ is determined uniquely by $\chi|_{N\times U}$. 
\end{proof}

\begin{defi}
For a group $X$ we define
\begin{equation}
e(X) = \frac{1}{|X|}\sum_{x\in X} x
\end{equation}
\end{defi}

\begin{defi}\label{defi MXpq}
	Assume that $p$ is a prime dividing the order of $N$ and let $q$ be any prime not dividing the order of $N_{p'}$ (in particular, $q=p$ is a possible choice for $q$).
	Let $X \leq N\times U$ be a subgroup such that $(N_{p'}\times U_{p'})/X_{p'}$ is cyclic. 
	Let $e$ denote the primitive idempotent in the rational group algebra $\Q ((N_{p'}\times U_{p'})/X_{p'})$ corresponding to the 
	unique faithful irreducible representation of $(N_{p'}\times U_{p'})/X_{p'}$ over $\Q$, and denote its preimage in
	$\Q (N_{p'}\times U_{p'})$ by $\widehat e$ (we may choose $\widehat e$ in such a way that $\widehat e \cdot e(X_{p'})=\widehat e$). 
	
	We define a $\Z_{(q)}(N\times U)$-lattice
	\begin{equation}
		M_0(X,p,q) = \Z_{(q)}\uparrow _{X_p}^{N_p\times U_p} \otimes \left( 
		\Z_{(q)}\uparrow _{X_{p'}}^{N_{p'}\times U_{p'}}\cdot \widehat e
		\right)
	\end{equation}
	as well as a $\Z_{(q)}(G\times U)$-lattice 
	\begin{equation}
		M(X,p,q) = M_0(X,p,q)\uparrow _{N\times U}^{G\times U}
	\end{equation}
\end{defi}

\begin{prop}\label{prop character M}
	The character of the $\Z_{(q)}(N\times U)$-lattice $M_0(X,p,q)$ is equal to
	$$
		\psi_{M_0(X,p,q)}=1\uparrow_{X_p}^{N_p\times U_p} \otimes \varphi
	$$
	where $\varphi$ is the unique irreducible rational character of
	$N_{p'}\times U_{p'}$ with kernel $X_{p'}$. 
\end{prop}
\begin{proof}
	This follows immediately from the fact that $\varphi$ is afforded by the $\Z_{(q)}(N_{p'}\times U_{p'})$-lattice $\Z_{(q)}\uparrow_{X_{p'}}^{N_{p'}\times U_{p'}}\cdot \widehat e$.
\end{proof}

\begin{remark}\label{remark characters}
	The following description of the character of $M(X,p,q)$ for certain $X$ is useful for explicit computations, even though we do not use it in this article:  
	\begin{enumerate}
		\item If $X_{p'}=N_{p'}\times U_{p'}$, then
		\begin{equation}
			\psi_{M(X,p,q)}= 1\uparrow_{X}^{G\times U}
		\end{equation}
		\item If $(N_{p'}\times U_{p'})/X_{p'}$ is cyclic of order $r$ for some prime $r$, then
		\begin{equation}
		\psi_{M(X,p,q)} = 1\uparrow_{X}^{G\times U}-1\uparrow_{X\cdot (N_{p'}\times U_{p'})}^{G\times U} 
		\end{equation}
	\end{enumerate}
\end{remark}
\begin{prop}\label{prop projectivity}
	If $(X\cap G)_q=\{1\}$, then $M(X,p,q)|_G$ is projective. 
\end{prop}
\begin{proof}
	This follows from the Mackey formula, as $M(X,p,q)$ is a direct summand of $\Z_{(q)}\uparrow_{X}^{G\times U}$:
	\begin{equation}
		\Z_{(q)}\uparrow_{X}^{G\times U}|_G\cong \bigoplus_{(g,u)} 
		 \Z_{(q)}\uparrow_{X^{(g,u)}\cap G}^G = \bigoplus_{(g,u)} 
		\Z_{(q)}\uparrow _{(X\cap G)^{(g,u)}}^G 
	\end{equation}
	where the summation index $(g,u)$ runs over a transversal of the double cosets $X\setminus (G\times U)/G$. 
	Each summand on the right hand side is induced from a $q'$-subgroup of $G$, and therefore is projective.
\end{proof}

\begin{lemma}\label{lemma extending chin}
	Let $p$ be a prime dividing the order of $N$ and let $n \in N_p$ be some $p$-element of $N$.
	Let $\chi_n$ be the following character of $N\times U$:
	\begin{equation}
	\chi_n = \sum_{C_G(n)\cdot g \in C_G(n)\setminus G} 1\uparrow_{[n^g]_p}^{N_p\times U_p} \otimes \xi_{n}^g
	\end{equation}
	where
	$\xi_{n}$ is the character of $N_{p'}\times U_{p'}$ defined in the beginning of this section (in particular $\xi_n$ is stabilised by $C_G(n)$). 
	
	Then, for any prime $q$ not dividing the order of $N_{p'}$, $\chi_n$ is the restriction to $N\times U$ of the character of the
	$\Z_{(q)} (G\times U)$-lattice
	\begin{equation}\label{eqn formula for M}
	L =\bigoplus_{\varphi \in {\rm Irr}_{\Q}(N_{p'}\times U_{p'})/{C_G(n)}} M([n]_p\times {\rm Ker}(\varphi), p, q)^{\oplus \mu(\varphi, n)}
	\end{equation}
	with
	\begin{equation}
	\mu(\varphi,n) = \frac{(\varphi, \xi_{n})}{(\varphi,\varphi)} \cdot \frac{1}{[C_G(n)\cap N_G({\rm Ker}(\varphi)):N]} \in \Z_{\geq 0}
	\end{equation}
	Moreover, the restriction of $L$ to $G$ is a projective $\Z_{(q)}G$-lattice.
\end{lemma}
\begin{proof}
	We need to prove three things:
	\begin{enumerate}
		\item The $\mu(\varphi, n)$ as defined above are (non-negative) integers.
		\item The lattice $L$ defined in Lemma~\ref{eqn formula for M} restricted to $G$ is projective.
		\item The restriction to $N\times U$ of the character of $L$ is equal to $\chi_n$.
	\end{enumerate}
	
	Recall that
	$$
		\xi_n = \sum_{m \in N_{p'}} \varepsilon_{(m\cdot n)^G}\OfU\cdot 1\uparrow_{[m]_{p'}}^{N_{p'}\times U_{p'}}
	$$
	Take some $\varphi\in \Irr_{\Q}(N_{p'}\times U_{p'})$. Then $\varphi$ is the sum over the Galois conjugacy class of
	some $\varphi_0\in \Irr_{\C}(N_{p'}\times U_{p'})$. In particular, $(\varphi,\varphi)=\varphi(1)$ and 
	$(\varphi, \xi_n)=\varphi(1)\cdot (\varphi_0,\xi_n)$, since $\varphi_0$ is linear and $\xi_n$ rational. 
	By Frobenius reciprocity
	\begin{align}\label{eqn xiphitrick}
	(\varphi_0, \xi_{n}) =& \sum_{m \in N_{p'}} \varepsilon_{(m\cdot n)^G}\OfU\cdot \left(\varphi_0, 1\uparrow_{[m]_{p'}}^{N_{p'} \times U_{p'}}\right) \\
	=& \sum_{m \in N_{p'}} \varepsilon_{(m\cdot n)^G}\OfU \cdot \left(\varphi_0|_{[m]_{p'}}, 1_{[m]_{p'}}\right).
	\end{align}
	Now,
	$$ 
		\left(\varphi_0|_{[m]_{p'}}, 1_{[m]_{p'}}\right) = \left\{ \begin{array}{ll} 1 & \textrm{ if } (m, c_{p'}) \in {\rm Ker}(\varphi_0)={\rm Ker}(\varphi), \\ 0 & \textrm{ otherwise } \end{array} \right. 
	$$
	So for $g \in C_G(n) \cap N_G({\rm Ker}(\varphi))$ we have
	$$
	 \varepsilon_{(m\cdot n)^G}\OfU \cdot \left(\varphi_0|_{[m]_{p'}}, 1_{[m]_{p'}}\right) = \varepsilon_{(m^g\cdot n)^G}\OfU \cdot \left(\varphi_0|_{[m^g]_{p'}}, 1_{[m^g]_{p'}}\right)
	$$
	Hence grouping together elements conjugate by $C_G(n) \cap N_G({\rm Ker}(\varphi))$ we can write
	$$
	(\varphi_0, \xi_{n}) = \sum_{\substack{m^{C_G(n) \cap N_G({\rm Ker}(\varphi))}, \\ m \in N_{p'}}} [C_G(n) \cap N_G({\rm Ker}(\varphi)):N] \cdot \varepsilon_{(m\cdot n)^G}\OfU\cdot \left(\varphi_0|_{[m]_{p'}}, 1_{[m]_{p'}}\right)
	$$
	where we use our assumption that $C_G(n)/N$ acts semiregularly on $m^G$ whenever $\varepsilon_{(m\cdot n)^G}\OfU \neq 0$ (that is Condition~\ref{cond cent n eq N}). 
	It follows that $\mu(\varphi,n)$ is an integer.
	
	The fact that $L|_G$ is projective follows immediately from Proposition \ref{prop projectivity}, since $[n]_p \cap G = \{1\}$ and for each $\varphi$
	\begin{equation}
	([n]_p \times {\rm Ker}(\varphi))\cap G = ([n]_p\cap G) \times ({\rm Ker}(\varphi)\cap G) = {\rm Ker}(\varphi) \cap G 
	\end{equation}
	is a subgroup of $N_{p'}$, and therefore a $q'$-group.
	
	Now let us prove that the character of $L|_{N\times U}$ is equal to $\chi_n$.
	Recall from Proposition \ref{prop character M} that for any $\varphi \in \Irr_{\Q}(N_{p'}\times U_{p'})$ and any $g\in G$
	\begin{equation}
		\psi_{M_0([n^g]_p\times {\rm Ker}(\varphi), p, q)}=1\uparrow_{[n^g]_p}^{N_p\times U_p} \otimes \varphi
	\end{equation}
	We can therefore write $\chi_n$ as follows:
	\begin{equation}
	\begin{array}{rcl}
	\chi_n&=& \displaystyle\frac{1}{[C_G(n):N]}\sum_{gN\in G/N} 1\uparrow_{[n^g]_p}^{N_p\times U_p} \otimes \xi_{n}^g		\\\\
	&=& \displaystyle\displaystyle\frac{1}{[C_G(n):N]}\sum_{\varphi \in {\rm Irr}_\Q(N_{p'}\times U_{p'})} \frac{(\varphi, \xi_{n})}{(\varphi,\varphi)} \sum_{gN\in G/N} \psi_{M_0([n^g]_p\times {\rm Ker}(\varphi)^g,p,q)}\\\\
	&=& \displaystyle\displaystyle\frac{1}{[C_G(n):N]}\sum_{\varphi \in {\rm Irr}_\Q(N_{p'}\times U_{p'})} \frac{(\varphi, \xi_{n})}{(\varphi,\varphi)}\cdot   \psi_{M([n]_p\times {\rm Ker}(\varphi),p,q)}|_{N\times U}\\\\
	&=& \displaystyle\displaystyle\sum_{\varphi \in {\rm Irr}_\Q(N_{p'}\times U_{p'})/{C_G(n)}} \frac{(\varphi, \xi_{n})}{(\varphi,\varphi)}\cdot  \frac{[C_G(n):C_G(n)\cap N_G({\rm Ker}(\varphi))]}{[C_G(n):N]}\cdot \psi_{M([n]_p\times {\rm Ker}(\varphi),p,q)}|_{N\times U}\\\\
	&=& \displaystyle\displaystyle\sum_{\varphi \in {\rm Irr}_\Q(N_{p'}\times U_{p'})/{C_G(n)}} \frac{(\varphi, \xi_{n})}{(\varphi,\varphi)}\cdot  \frac{1}{[C_G(n)\cap N_G({\rm Ker}(\varphi)):N]} \cdot \psi_{M([n]_p\times {\rm Ker}(\varphi),p,q)}|_{N\times U}\\\\
	&=& \displaystyle\displaystyle\sum_{\varphi \in {\rm Irr}_\Q(N_{p'}\times U_{p'})/{C_G(n)}} \mu(\varphi, n)\cdot \psi_{M([n]_p\times {\rm Ker}(\varphi),p,q)}|_{N\times U}		
	\end{array}
	\end{equation} 
	The latter is clearly the character of $L|_{N\times U}$. Going from the third to the fourth line we made use of the fact
	that $(\varphi, \xi_{n})$ is constant on the $C_G(n)$-orbit of $\varphi$, as $\xi_{n}$ is assumed to be $C_G(n)$-invariant.
\end{proof}

We need one last proposition before proceeding to the proof of Theorem \ref{thm existence u semilocal}.
\begin{prop}\label{prop induced characters determined}
	Let $X$ be a group and let $Y\trianglelefteq X$ be a normal subgroup.
	If $\chi$ and $\psi$ are virtual characters of $Y$, then $\chi\uparrow^X=\psi\uparrow^X$ if and only if
	$\chi\uparrow^X|_Y=\psi\uparrow^X|_Y$.
\end{prop}
\begin{proof}
	We claim that $\chi\uparrow^X|_Y\uparrow^X = [X:Y]\cdot \chi\uparrow^X$ (same for $\psi$). This would clearly imply the assertion, 
	and it is an easy application of the Mackey formula (note that $Y\setminus X/Y=X/Y$ in this case):
	\begin{equation}
	(\chi\uparrow^X|_Y)\uparrow^X = \left(\sum_{xY} \chi^x|_{Y^x\cap Y}\uparrow^Y\right)\uparrow^X =\sum_{xY} \chi^x\uparrow^X = [X:Y]\cdot \chi\uparrow^X
	\end{equation}
\end{proof}

We now prove the main result of this section.

\begin{proof}[Proof of Theorem \ref{thm existence u semilocal}]
	Let $p$ be a prime dividing the order of $N$. For $n\in N_p$ define
	$$
		\chi_n = \sum_{C_G(n)\cdot g \in C_G(n)\setminus G} 1\uparrow_{[n^g]_p}^{N_p\times U_p} \otimes \xi_{n}^g
	$$
	By formula \eqref{eqn decomp xin} we then have
	\begin{equation}
		\chi|_{N\times U} = \sum_{n^G, n \in N_p} \chi_n
	\end{equation}
%	where $n$ ranges over a set of representatives of the conjugacy classes of $G$ contained in $N$.
	By applying Lemma \ref{lemma extending chin} to the individual $\chi_n$ we get, for any prime $q$ not dividing the order of $N_{p'}$, a $\Z_{(q)} (G\times U)$-lattice $L$ such that $L|_{N\times U}$ has character $\chi|_{N\times U}$ and $L|_G$ is projective. Furthermore, since all summands of $L$ are induced from $N\times U$, and similarly $\chi$
	is induced from a (potentially virtual) character of $N\times U$, Proposition \ref{prop induced characters determined} implies that the character of $L$ is equal to $\chi$. In particular, this shows that $\chi$ is in fact a proper character of $G\times U$.
	
	Since we can do the above for all primes $p$ dividing the order of $N$, we do in fact get a $\Z_{(q)}(G\times U)$-lattice $L(q)$ with character $\chi$ and projective restriction to $G$ for all prime numbers $q$. 
	Moreover, Proposition \ref{prop sum induced partaug} ensures that $\chi$ is the character of 
	a $G$-regular $\Q (G\times U)$-module $V$, and that the partial augmentations of the associated unit are given by $\eps\OfU$. We may assume without loss that the $L(q)$ are lattices in $V$. Using Proposition \ref{prop identify local double action} it also follows that the $L(q)$ are $G$-regular.
	Now define
	$$
		L = \bigcap_{q\in \pi} L(q)
	$$ 
	Then $\Z_{(q)}L|_G = L(q)|_G\cong \Z_{(q)}G$ for all $q\in \pi$. Hence Proposition \ref{prop local to semilocal} implies that $L$ is a $G$-regular $\Z_{\pi}(G\times U)$ module with character $\chi$. This concludes the proof.
\end{proof}

\begin{remark}
	While it has no bearing on the proof of Theorem \ref{thm existence u semilocal}, it still seems worth pointing out 
	that for every prime $p$ dividing the order of $N$, there is at most one $G$-conjugacy class $n^G$ of elements in $N_{p}$ such that $\xi_n\neq 0$, or, equivalently, $\chi_n\neq 0$. 
	%This can be seen by evaluating the formula for $\xi_n$ given in Remark~\ref{remark formula xin} at the unit of $N$. 
	This can be seen by considering the degree of $\xi_n$, which can be computed using the formula given in Remark~\ref{remark formula xin}.
	What we obtain is that $\xi_n(1)$ is equal to $|N_{p'}|$ times the sum over a certain subset of the $\eps_{g^G}\OfU$. Since none of these sums can be negative, and the $\eps_{g^G}\OfU$ are integers summing up to one, it follows that at most one of these sums can be non-zero.
\end{remark}

In order to apply Theorem \ref{thm existence u semilocal} later on we will need to verify the condition that the $\xi_n$ are proper characters. The following lemma, which was also proved in \cite[Corollary 3.5]{MargolisdelRioCW1}, helps with that.
\begin{lemma}\label{ScalarProdXiPsi}
	Let $p$ be a prime dividing the order of $N$ and let $\varphi$ be an irreducible complex character of $N_{p'} \times U_{p'}$. 
	\begin{enumerate}
		\item If $\varphi$ is the trivial character then
		$$ ( \xi_n, \varphi) =  [C_G(n):N] \cdot \sum_{(n\cdot m)^G, m \in N_{p'}} \eps_{(n\cdot m)^G}\OfU$$
		\item If there is no $m_0 \in N_{p'}$ such that $(m_0, c_{p'}) \in \Ker(\varphi)$ then
		$$ ( \xi_n, \varphi) = 0 $$
		\item Otherwise set $K = \Ker(\varphi) \cap N_{p'}$ and let $m_0 \in N_{p'}$ be chosen such that $(m_0, c_{p'})\in \Ker(\varphi)$. Then
		$$  ( \xi_n, \varphi) =  \sum_{m^{C_G(n)}, m \in N_{p'}} \left|m^{C_G(n)} \cap m_0\cdot K\right| \cdot \varepsilon_{(n\cdot m)^G}\OfU $$
	\end{enumerate}
	
	In particular, $\xi_n$ is a proper character of $N_{p'}\times U_{p'}$ if and only if for all subgroups $K$ of $N_{p'}$ such that $N_{p'}/K$ is cyclic and for all $m_0 \in N_{p'}$ we have 
	$$\sum_{m^{C_G(n)}, m\in N_{p'}} \left|m^{C_G(n)} \cap m_0\cdot K\right|\cdot  \varepsilon_{(n\cdot m)^G}\OfU \geq 0$$
\end{lemma}
\begin{proof}
	As in formula \eqref{eqn xiphitrick} we have
	\begin{equation}
	\label{eqn xiphitrick2}
	(\xi_n, \varphi) = \sum_{m \in N_{p'}} \varepsilon_{(n\cdot m)^G}\OfU\cdot \left(\varphi|_{[m]_{p'}}, 1_{[m]_{p'}}\right) 
	\end{equation}
	and $(\varphi|_{[m]_{p'}}, 1_{[m]_{p'}})$ is equal to one if $(m,c_{p'})\in \Ker(\varphi)$ and equal to zero otherwise.
	
	So for $\varphi$ equal to the the trivial character we have
	$$ (\xi_n, \varphi) = \sum_{m \in N_{p'}} \varepsilon_{(n\cdot m)^G}\OfU = \sum_{(n\cdot m)^G, m \in N_{p'}} \left|m^{C_G(n)}\right|\cdot \varepsilon_{(n\cdot m)^G}\OfU = [C_G(n):N] \sum_{(n\cdot m)^G, m \in N_{p'}} \varepsilon_{(n\cdot m)^G}\OfU $$
	where we used that the centraliser of $m$ in $C_G(n)$ is equal to $C_G(m)\cap C_G(n)=N$ whenever $\eps_{(n\cdot m)^G}\OfU\neq 0$.
	
	If there is no $m_0 \in N_{p'}$ such that $(m_0,c_{p'})\in \Ker(\varphi)$ then all summands on the right hand side of \eqref{eqn xiphitrick2} are zero, which implies the second assertion.
	
	Finally the third case follows directly by grouping together summands in \eqref{eqn xiphitrick2} for which $m$ is in the same $C_G(n)$-conjugacy class. All one has to use here is that an $m\in N_{p'}$ satisfies $(m,c_{p'})\in\Ker(\varphi)$ if and only if $m\in m_0\cdot K$ by definition of $m_0$ and $K$.
\end{proof}

\section{A local-global principle for certain torsion units}\label{section existence global}

In this section we will show that by making only slightly stronger assumptions on $G$ and $\chi$, the semi-local units $u_\pi\in \UU(\Z_\pi G)$ constructed in Theorem \ref{thm existence u semilocal} 
can be shown to be conjugate to elements of $\UU(\Z G)$. This follows from a general local-global principle which might also prove useful for other problems which have a ``double action'' formulation (such as subgroup conjugation questions for $\UU (\Z G)$). In essence, the argument boils down to the following: if $u\in \UU(\Q G)$ is a torsion unit which has an eigenvalue equal to one in each simple component of $\Q G$, then any unit in $\UU(Z(\Q_pG))$ (for any $p$) can be realised as the reduced norm of an element of the centraliser of $u$ in $\UU(\Q_pG)$. It follows that if $u$ is conjugate to an element of $\UU(\Z_p G)$, then it can be conjugated into $\UU(\Z_p G)$ by means of an element of reduced norm one. This holds true for all $p$, and in this situation the strong approximation theorem for the kernel of the reduced norm (see \cite[Theorem 51.13]{CurtisReinerII}) guarantees the existence of an element in $\UU(\Q G)$ of reduced norm one which conjugates $u$ into $\Z_p G$ for all $p$ simultaneously, that is, it conjugates $u$ into $\Z G$.  

\begin{lemma}\label{lemma localglobal}
	Let $R$ be the ring of integers in an algebraic number field $K$, let $B$ be a finite dimensional $K$-algebra and let $A\subseteq B$ be a semisimple $K$-subalgebra of $B$ satisfying the Eichler condition relative to $R$. Moreover, let $\Lambda$ be an $R$-order in $A$ and let $\Gamma$ be an $R$-order in $B$ containing $\Lambda$.
	By $\pi$ we denote the set of maximal ideals $p$ of $R$ such that $\Lambda_p$ is not a maximal order, and we assume that
	 $V$ is a $B$-module  such that
	\begin{enumerate}
	\item $V|_A$ is free of rank one as an $A$-module.
	\item There is an idempotent $e\in \End_A(V|_A)$ such that $e\cdot \End_A(V|_A) \cdot e \subseteq \End_B(V)$ and
	$e\cdot \eta \neq 0$ for all primitive idempotents $\eta \in Z(\End_A(V|_A))$.
	\end{enumerate} 
	Then our claim is the following: for every $R_\pi\Gamma$-lattice $L(\pi)\leq V$ such that $L(\pi)|_{R_\pi\Lambda}$ is free of rank one as a $R_\pi\Lambda$-module there is a $\Gamma$-lattice $L\leq V$ such that $L|_\Lambda$ is free of rank one as a $\Lambda$-module and $R_\pi \cdot L \cong L(\pi)$. 
\end{lemma}
\begin{proof}
	Fix an isomorphism of right $A$-modules $\varphi:\ V|_A \stackrel{\sim}{\longrightarrow} A$. We may identify $\End_A(A)$ with $A$, where $a\in A$ is identified with the endomorphism of $A$ induced by left multiplication by $a$
	(our notational conventions ensure that we do not have to consider the opposite ring of $A$ here, as one is often compelled to do in similar situations). Hence $\alpha \mapsto \varphi\circ \alpha\circ \varphi^{-1}$ induces an isomorphism between $\End_A(V|_A)$ and $\End_A(A)=A$. Let $f$ denote the image of $e$ under this isomorphism.
	Then the algebra $C=fAf$ is contained in the image of $\End_{B}(V)$, and $C$ is again a semisimple $K$-algebra with the additional property that $C\cdot \eta \neq \{0\}$ for all primitive idempotents $\eta\in Z(A)$. The latter ensures that the map $Z(A)\longrightarrow Z(C):\ z \mapsto z\cdot f$ is an isomorphism. If $p$ is a maximal ideal of $R$, we also have that $C_p$ is contained in the image of $\End_{B_p}(V_p)$, and multiplication by $f$ again induces an isomorphism between $Z(A)_p$ and $Z(C)_p$ (note that $Z(A_p)=Z(A)_p$, and the same holds for $C$). Moreover, it follows immediately from the definition of reduced norms that $\nr_{C_p/Z(C_p)}(c)=\nr_{A_p/Z(A_p)}((1-f)+c)\cdot f$ for any $c\in \UU(C_p)$.
	
	By \cite[Theorem 7.45]{CurtisReinerI} $$\nr_{C_p/Z(C_p)}(\UU(C_p))=\UU(Z(C_p))$$ for each maximal ideal $p$ of $R$. Hence we can
	find, for each $a\in \UU(Z(A_p))$, a $c\in \UU(C_p)$ such that $$\nr_{A_p/Z(A_p)}((1-f)+c)=a$$ Of course, 
	the element $(1-f)+c$ also lies in the image of $\End_{B_p}(V_p)$.
	
	Next let us pick an arbitrary $\Gamma$-lattice $L'\leq V$ with the property that $R_{\pi}L'=L(\pi)$ (for instance, we could take $L'$ to be the $\Gamma$-lattice generated by some $R_\pi\Gamma$-generating set of $L(\pi)$).
	Then for each prime $p\in \pi$ the completion $L'_p$ is isomorphic to $(L(\pi))_p$, which is free of rank one as a $\Lambda_p$-lattice by definition of $L(\pi)$. For every $p\not\in\pi$ the order $\Lambda_p$ is maximal, and therefore $L'_p$ restricted to $\Lambda_p$ is free of rank one since 
	$K_pL'$ restricted to $A_p$ is free of rank one (this is by virtue of \cite[Theorem 18.10]{Reiner}). We conclude that $L'$ restricted to $\Lambda$ is locally free. 
	Therefore we can write
	$$
		\varphi(L') = \alpha' \Lambda = \bigcap_p \alpha'_p \Lambda_p
	$$ 
	for some id\`ele $\alpha' = (\alpha_p')_p \in J(A)$. Since $\alpha'_p\in \UU(\Lambda_p)$ for all except finitely many $p$, we may as well assume that $\alpha'_p=1$ for all except finitely many $p$. By the arguments above we can find elements $c_p\in C_p$ (one for each maximal ideal $p$ of $R$)
	such that $\nr_{C_p/Z(C_p)}(c_p)=\nr_{A_p/Z(A_p)}(\alpha'_p)\cdot f$. We can assume without loss that $c_p=f$ whenever $\alpha'_p=1$. Then $\alpha = (\alpha_p)_p=(((1-f)+c_p)^{-1}\cdot \alpha'_p)_p$ is an element of $J_0(A)$, which means that
	$$
		L = \varphi^{-1}(\alpha\Lambda)
	$$  
	is stably free by Theorem \ref{thm froehlich}. Since $A$ satisfies the Eichler condition relative to $R$ we know that $L$ is free of rank 1 as a $\Lambda$-module by Theorem~\ref{thm jacobinski}. All we need to show now is that $L$ is a $\Gamma$-lattice and $L_p\cong (L(\pi))_p$ for all $p\in \pi$. But, for any maximal ideal $p$ of $R$, multiplication by $((1-f)+c_p)$ from the left induces an isomorphism between $\alpha_p\Lambda_p$ and $\alpha'_p\Lambda_p$. By definition, $((1-f)+c_p)$ lies in the image of $\End_{B_p}(V_p)$. That is, there is a $\gamma_p\in\End_{B_p}(V_p)$ such that $\gamma_p(L_p)=\gamma_p(\varphi^{-1}(\alpha_p\Lambda_p))=\varphi^{-1}(((1-f)+c_p)\cdot \alpha_p\Lambda_p)=\varphi^{-1}(\alpha'_p\Lambda')=L'_p$. This shows that each $L_p$ is a $\Gamma_p$-lattice of the desired isomorphism type, and since $L$ is the intersection of the $L_p$'s, it also follows that $L$ is a $\Gamma$-lattice.
\end{proof}

\begin{thm}\label{thm existence u global}
	Assume we are in the setting of Theorem~\ref{thm existence u semilocal}, and
	suppose that the Conditions \ref{cond sum 1}, \ref{cond cent n eq N} and \ref{cond xi proper char} hold. If in addition to that the following conditions are satisfied:
	\begin{enumerate}[label=\textbf{(C.\arabic*})]
	\setcounter{enumi}{3}
	\item\label{cond eichler} $G$ does not have an epimorphic image isomorphic to either 
	one of the following: 
	\begin{enumerate}
	\item A generalised quaternion group of order $4n$ where $n\geq 2$.
	\item The binary tetrahedral group of order $24$.
	\item The binary octahedral group of order $48$.
	\item The binary icosahedral group of order $120$.
	\end{enumerate}
	\item\label{cond eigenvalue} $(\chi, \eta\otimes 1_U) \neq 0$ for every $\eta\in \Irr_{\C}(G)$.
	\end{enumerate}
	Then there exists a $G$-regular $\Z (G \times U)$-lattice $L$ with character $\chi$. The partial augmentations of the associated unit $u\in \UU(\Z G)$ are given by $\eps\OfU$.
\end{thm}
\begin{proof}
	Theorem~\ref{thm existence u semilocal} ensures that there is a $G$-regular $\Z_\pi (G\times U)$-lattice $L(\pi)$ 
	in a $\Q (G\times U)$-module $V$
	with character $\chi$, where $\pi$ is the set of all prime divisors of the order of $G$.
	Our assertion will follow once we show that there is a $G$-regular $\Z (G\times U)$-lattice $L\leq V$ with $\Z_\pi L \cong L(\pi)$. Note that by Proposition~\ref{prop double action iso implies conj} we may assume without loss that $L(\pi) = {_{u_\pi}(\Z_\pi G)_G}$ and $V={_{u_\pi}(\Q G)_G}$ for a unit $u_\pi \in \UU(\Z_\pi G)$ of order $n$.
	
	Now if $\eta$ is a primitive idempotent in $Z(\Q G)$ corresponding to a character $\varphi \in \Irr_\Q(G)$, then
	$$
		\eta_{G\times U} =\eta \cdot \frac{1}{n}\sum_{i=1}^n c^i \in Z(\Q (G\times U))
	$$
	is the primitive idempotent in $Z(\Q (G\times U))$ belonging to the character $\varphi \otimes 1_U$. Since $(\chi, \varphi_0\otimes 1_U)\neq 0$ for all irreducible complex characters $\varphi_0$ occurring in $\varphi$, it follows that $V\cdot \eta_{G\times U}\neq 0$.
	
	Using the fact that the action of $G\times U$ on $V={_{u_\pi}(\Q G)_G}$ is given explicitly, we get 
	\begin{equation}\label{eqn eps e neq 0}
		\{0\} \neq V\cdot \eta_{G\times U} = \left(\frac{1}{n}\sum_{i=1}^n (u_\pi^\circ)^i\right) \cdot \Q G \cdot \eta = \eta \cdot \left(\frac{1}{n}\sum_{i=1}^n (u_\pi^\circ)^i\right) \cdot \Q G  
	\end{equation}
	Now define
	$$
		e = \frac{1}{n} \sum_{i=1}^n(u_\pi^\circ)^i
	$$
	and $C=e\Q Ge$.  Clearly, left multiplication by elements of $C$ commutes with left multiplication by $u_\pi^\circ$, that is, left multiplication by elements of $C$ induces $\Q (G\times U)$-module endomorphisms of $V$.
	Since $V$ restricted to $G$ is just $\Q G$ viewed as a right module over itself, we may identify 
	$\Q G$ with $\End_{\Q G}(V|_G)$. Concretely, an element $a\in \Q G$ may be identified with the $\Q G$-endomorphism of $V$ induced by left multiplication with $a$. To summarise, what we have shown is that
	$$
		e\cdot \End_{\Q G}(V|_G) \cdot e \subseteq \End_{\Q (G\times U)}(V)
	$$
	and $e\cdot \eta \neq 0$ for all primitive idempotents $\eta$ in $\End_{\Q G}(V|_G)$ by \eqref{eqn eps e neq 0} above. 
	Moreover, by Theorem~\ref{thm qg eichler} our first condition implies that $\Q G$ satisfies the Eichler condition relative to $\Z$. 
	Hence we may apply Lemma \ref{lemma localglobal} to obtain a $G$-regular $\Z (G\times U)$-lattice $L\leq V$ with $\Z_\pi L\cong L(\pi)$. This completes the proof. 
\end{proof}

\section{The counterexample}

We will now restrict our attention to a specific family of metabelian groups, consisting of the groups $G(p,q;d;\alpha,\beta)$ defined in the introduction, where the parameters $p$ and $q$ are two different primes, $d$  is a common divisor of $p^2-1$ and $q^2-1$ which divides neither $p+1$ nor $q+1$, and $\alpha$ and $\beta$ are primitive elements in $\F_{p^2}$ and $\F_{q^2}$, respectively. 
Groups of this type were recently studied, in a related context, in \cite{MargolisdelRioCW3}. This work provided the motivation to look at these groups as potential counterexamples to the Zassenhaus conjecture.

Our first aim is to reformulate the conditions under which Theorems \ref{thm existence u semilocal} and \ref{thm existence u global} yield semi-local and global units, respectively, in elementary terms for the $G(p,q;d;\alpha,\beta)$'s. This reformulation is stated in Theorem~\ref{thm zc Gpqd} below. The proof of this theorem is spread out over several propositions and lemmas, each corresponding, more or less, to one of the conditions of Theorems \ref{thm existence u semilocal} and \ref{thm existence u global}. The proofs of the main theorems of this article, the fact that $G(7,19;3;\alpha,\beta)$ (with $\alpha$ a root of $X^2-X+3$ over $\mathbb{F}_7$ and $\beta$ a root of $X^2-X+2$ over $\mathbb{F}_{19}$) is a counterexample to the Zassenhaus conjecture and so are infinitely many more $G(p,q;d;\alpha, \beta)$, are then a quick application of the aforementioned Theorem~\ref{thm zc Gpqd}.

%\begin{defi}\label{defi Gpqd}
%Assume that $p$ and $q$ are two different prime numbers, 
%and assume that $d$ is a common divisor of $p^2-1$ and $q^2-1$ such that $\gcd(d,p+1)=\gcd(d,q+1)=1$. Let $\alpha$ and $\beta$ be primitive elements of $\mathbb{F}_{p^2}$ and $\mathbb{F}_{q^2}$ respectively.
%Let $N$ be the underlying additive group of $\mathbb{F}_{p^2}\times \mathbb{F}_{q^2}$,
%and consider the automorphisms $a,b$ and $c$ of $N$ given by
%$$(x,y)^a = (\alpha^d \cdot x,y), \quad (x,y)^b = (x,\beta^d \cdot y), \quad (x,y)^c = (\alpha\cdot x, \beta\cdot  y),$$
%for $x\in \mathbb{F}_{p^2}$ and $y\in \mathbb{F}_{q^2}$.
%Let $A = \langle a,b,c \rangle$.
%Then $A$ is the abelian group given by the following presentation (as an abelian group)
%$$A = \langle a,b,c\ | \ a^{\frac{p^2-1}{d}}=b^{\frac{q^2-1}{d}}=1, c^d=a\cdot b \rangle.$$
%We define 
%$$
%	G(p,q;d;\alpha,\beta)  = N \rtimes A
%$$
%\end{defi}

Whenever we use the group $G(p,q;d;\alpha,\beta)$ below we will tacitly assume the entire notation used in the definition of this group, i.e.  the subgroups $N$ and $A$, as well as the generators $a,b,c$ of $A$. Since the definition of $G(p,q;d;\alpha,\beta)$ is symmetric in $p$ and $q$ (of course interchanging $\alpha$ and $\beta$ as well) all statements we make below have an analogue with the roles of $p$ and $q$ reversed. We do not always state this analogue explicitly.

We will often use the fact that $$|A|=\frac{(p^2-1)\cdot (q^2-1)}{d}$$ as well as the facts that $C_A(N_p)=\langle b \rangle$ and $C_A(N_q)=\langle a \rangle$. Moreover $G(p,q;d;\alpha,\beta)$ is a metabelian group.

\begin{notation}\label{notation K_p, r_i}
	For $G = G(p,q;d;\alpha,\beta)$ we define the following subset of  $\F_{p^2}\times\{0\} = N_p$
	$$K_p = \{(\alpha + x,0) \ | \ x \in \F_p \}$$
	Moreover, for any $i\in \Z$, set 
	$$r_i(p) = \left| \left\{ 1 \leq t \leq \frac{p^2-1}{d} \ \bigg| \ (\alpha^{i+t\cdot d},0) \in K_p\right\} \right|$$
	Notice that $r_i(p)=r_j(p)$ if $i\equiv j \bmod d$.
\end{notation}

The first goal of this section is to prove the following theorem:
\begin{thm}\label{thm zc Gpqd}
	Let $G=G(p,q;d;\alpha,\beta)$ and let $\eps\OfU: G \longrightarrow \Z:\ g \mapsto \eps_{g^G}\OfU$ be a class function such that
	\begin{enumerate} 
		\item $
		\sum_{g^G} \eps_{g^G}\OfU=1
		$
		\item If $\eps_{g^G}\OfU\neq 0$ for some $g\in G$, then $g\in N$ and the order of $g$ is $p\cdot q$
		\item For every $j\in \{0,\ldots,d-1\}$ the inequalities
		\begin{equation}\label{eqn inequality rp}
		\sum_{i=1}^{d} r_{j+i}(p)\cdot \eps_{(\alpha^i,1)^G}\OfU \geq 0 
		\end{equation}
		and 
		\begin{equation}\label{eqn inequality rq}
		\sum_{i=1}^{d} r_{j+i}(q)\cdot \eps_{(1,\beta^i)^G}\OfU \geq 0
		\end{equation}
		hold.
	\end{enumerate}
	Then there is a unit $u\in \UU(\Z G)$ of order $p\cdot q$ whose partial augmentations are given by the class function $\eps\OfU$. If $\eps_{(\alpha^i,1)^G}\OfU\neq 0$ for more than one $i\in\{1,\ldots, d\}$, then $u$ is not conjugate in $\UU(\Q G)$ to an element of the form $\pm g$ with $g\in G$.
\end{thm}
Once we have done that we will verify the conditions of this theorem for one concrete choice of values of $p$, $q$ and $d$ and a concrete class function $\eps\OfU$. That bit, which is of course at the same time the proof of Theorem A, is ultimately just a simple calculation (albeit a tedious one). The proof of Theorem B is an application of the following corollary.

\begin{corollary}\label{corollary infinite series}
	Fix an $M \in \N$ and let $G=G(p,q;d;\alpha,\beta)$. Assume that both $p$ and $q$ are greater than or equal to
	$$\frac{d^4\cdot M^2}{1- |\cos (2\pi/d)|}$$
	Let $\eps\OfU: G \longrightarrow \Z:\ g \mapsto \eps_{g^G}\OfU$ be a class function such that
	\begin{enumerate} 
		\item $
		\sum_{g^G} \eps_{g^G}\OfU=1
		$
		\item $|\eps_{g^G}|\leq M$ for all $g\in G$.
		\item If $\eps_{g^G}\OfU\neq 0$ for some $g\in G$, then $g\in N$ and the order of $g$ is $p\cdot q$.	
	\end{enumerate}	
	Then there is a unit $u\in \UU(\Z G)$ of order $p\cdot q$ whose partial augmentations are given by $\eps\OfU$.
\end{corollary}
\begin{proof}
	We just need to check that the inequalities \eqref{eqn inequality rp} and \eqref{eqn inequality rq} are satisfied.
	The situation is symmetric in $p$ and $q$, so we will just prove that \eqref{eqn inequality rp} holds. For brevity write $r_i$ instead of $r_i(p)$.
	
	First note that $\alpha^{p+1}$ is a primitive element of $\F_p$.
	 Let $\zeta_d$ be a primitive $d$-th root of unity in $\C$, and
	define a multiplicative character 
	$$\chi:\ \F_p \longrightarrow \Q(\zeta_d):\ \alpha^{p+1} \mapsto \zeta_d$$
	where we adopt the convention $\chi(0)=0$. Set $f(X)=X^2+(\alpha+\alpha^p)\cdot X + \alpha^{p+1}\in \F_p[X]$.
	Then, for any $x\in \F_p$, we have $\chi(f(x))=\chi((\alpha+x)\cdot (\alpha+x)^p)=\chi((\alpha+x)^{p+1})$. 
	So, if $\alpha+x$ can be written as $\alpha^i$ for some $i$, then $\chi(f(x))=\chi((\alpha^i)^{p+1})=\zeta_d^i$. By definition, $r_i$ is the number of $x\in \F_p$ such that $\alpha +x = \alpha^{i+t\cdot d}$ for some $t\in \Z$. Hence
	$r_i$ is exactly the number of $x\in \F_p$ such that $\chi(f(x))=\zeta_d^i$. This means that
	$$
		\sum_{x\in \F_p} \chi(f(x))=\sum_{i=1}^d r_i\cdot \zeta_d^i
	$$
	On the other hand, by \cite[Theorems 5.39 and 5.40]{Lidl}, the left hand side of this equation is equal to a 
	complex number $\omega$ of absolute value $\sqrt{p}$. If we write $\delta_i = r_i-\frac{p}{d}$
	and use that fact that $\zeta_d+\zeta_d^2+\ldots+\zeta_d^d=0$, we get 
	$$
		\omega = \sum_{i=1}^d \delta_i\cdot \zeta_d^i
	$$
	and thus
	$$
	\quad p=\left(\sum_{i=1}^d \delta_i\cdot \zeta_d^i\right)\cdot \left(\sum_{i=1}^d \delta_i\cdot \zeta_d^{-i}\right)=\sum_{i=1}^d \delta_i^2 +\sum_{1\leq i < j \leq d} (\zeta_d^{i-j}+\zeta_d^{j-i})\cdot \delta_i\cdot \delta_j
	$$
	Note that $\zeta_d^{i-j}+\zeta_d^{j-i}=2\cdot {\rm Re}(\zeta_d^{i-j})$, and the absolute value of this number is bounded above by $2\cdot |\cos (2\pi/d)|$. Hence
	$$
		p \geq \sum_{i=1}^d \delta_i^2 - 2\cdot |\cos(2\pi/d)| \cdot \left|\sum_{1\leq i < j \leq d} \delta_i\cdot \delta_j\right| = 
		\left(1-|\cos(2\pi/d)|\right)\cdot \sum_{i=1}^d \delta_i^2	$$
	In the second step we used the fact that $\delta_1+\ldots+\delta_d=0$ (a consequence of $r_1+\ldots+r_d=p$). We conclude that
	$$
		\delta_i \leq \sqrt{\frac{p}{1-|\cos(2\pi/d)|}}
	$$
	Now, for each $j\in\{0,\ldots,d-1\}$, the left hand side of the inequality \eqref{eqn inequality rp} can be bounded below as follows:
	$$
		\sum_{i=1}^d r_{j+i}\cdot \eps_{(\alpha^i,1)^G}= \frac{p}{d} + \sum_{i=1}^d \delta_{j+i} \cdot \eps_{(\alpha^i,1)^G} \geq \frac{p}{d} -d\cdot  \sqrt{\frac{p}{1-|\cos(2\pi/d)|}} \cdot M
	$$
	Our assumed lower bound on $p$ ensures that the right hand side of this is non-negative, which proves that the inequality \eqref{eqn inequality rp} is satisfied for each $j$.
\end{proof}
\begin{remark}
	The combination of Corollary \ref{corollary infinite series} for fixed $d$ and Dirichlet's theorem on arithmetic progressions clearly provides an infinite number of counterexamples to the Zassenhaus conjecture, with arbitrary prescribed partial augmentations for the elements of order $p\cdot q$, and hence a proof of Theorem B. However, it does not yield counterexamples of particularly small order. The smallest choice of parameters for which Corollary \ref{corollary infinite series} applies is $M=1$, $d=3$, $p=163$ and $q=167$. The resulting counterexample has order
	$2^7\cdot 3^4 \cdot 7 \cdot 41 \cdot 83 \cdot 163^2\cdot 167^2$. 
\end{remark}

To prove Theorem~\ref{thm zc Gpqd} we first collect elementary properties of the group $G(p,q;d;\alpha,\beta)$.

\begin{prop}\label{prop all pq conjugate}
	Let $G = G(p,q;d;\alpha,\beta)$. Then the following hold:
	\begin{enumerate}
		\item\label{prop all pq conjugate, N selbstzentrali} Let $g \in N$ be of order $p\cdot q$. Then $C_G(g) = N$.
		\item\label{prop all pq conjugate, C(n_p)=C(N_p)} For a non-trivial element $n \in N_p$ we have $C_G(n) = C_G(N_p)$.
		\item\label{prop all pq conjugate, K'klassen} Representatives of the $G$-conjugacy classes of element of order $p\cdot q$ in $N$ are given by $(1,1)$, $(\alpha,1)$, $(\alpha^2,1), \ldots, (\alpha^{d-1},1)$. Moreover, for a non-trivial element $n \in N_q$ the $C_G(n)$-conjugacy classes of elements of order $p$ in $N$ are given by $(1,0)$, $(\alpha,0), \ldots, (\alpha^{d-1},0)$.
		\item\label{prop all pq conjugate, G trans auf pq} $G$ acts transitively (by conjugation) on the set of cyclic subgroups of order $p\cdot q$ in $N$. 
		\item\label{prop all pq conjugate, G/C semireg} $G/C_G(N_p)$ acts regularly on the set of non-trivial cosets of cyclic groups of order $p$ in $N_p$, that is, the set
		$$
			\left\{ n\cdot X \ | \ X\leq N_p \textrm{ has order $p$ and } n\in N_p, n \not\in X \right\}
		$$ 
		This set has cardinality $p^2-1$.
		\item\label{prop all pq conjugate, C_A(N_q) semireg} $C_A(N_q)$ acts semiregularly on the set of non-trivial cosets of cyclic groups of order $p$ in $N_p$.
		\item\label{prop all pq conjugate, C_G(N_q) trans} $C_G(N_q)$ acts transitively on the set of cyclic groups of order $p$ in $N_p$.
	\end{enumerate}
\end{prop}
\begin{proof}
\begin{enumerate}
\item[\eqref{prop all pq conjugate, N selbstzentrali}] Let $(\alpha^i, \beta^j) \in N$ be some element of order $p\cdot q$ and $r,s,t \in \mathbb{Z}$. Then
$$(\alpha^i, \beta^j)^{a^rb^sc^t} = (\alpha^i, \beta^j) \Leftrightarrow rd \equiv -t \bmod (p^2-1), \ sd \equiv -t \bmod (q^2-1).$$
This implies that $t$ is divisible by $d$ and hence $c^t \in \langle a, b \rangle$. But then also $a^rb^sc^t = 1$.

\item[\eqref{prop all pq conjugate, C(n_p)=C(N_p)}] This follows directly since multiplication is a regular action on $\mathbb{F}_{p^2}^\times$.

\item[\eqref{prop all pq conjugate, K'klassen}] First note that two elements of the form $(\alpha^i,1)$ and $(\alpha^j,1)$, for some $i,j \in \mathbb{Z}$, are $G$-conjugate if and only if they are $C_G((0,1))$-conjugate. This follows since $(0,1)$ is the $q$-part of both elements. But $C_G((0,1)) = \langle a \rangle$ and since for any $t \in \mathbb{Z}$ we have
$$(\alpha^i,1)^{a^t} = (\alpha^{i+d\cdot t},1)$$
we get that $(\alpha^i,1)$ and $(\alpha^j,1)$ are $G$-conjugate if and only if $i \equiv j \bmod d$. In particular the elements $(1,1),\ldots,(\alpha^{d-1},1)$ are pairwise non-conjugate and contain representatives of the conjugacy classes of all elements of the form $(\alpha^i, 1)$. Now any element of order $p\cdot q$ in $N$ is of the form $(\alpha^j,\beta^k)$, for certain $j$ and $k$, and since $(\alpha^j,\beta^k)^{c^{-k}} = (\alpha^{j-k},1)$, any element of order $p\cdot q$ in $N$ is conjugate to an element of the form $(\alpha^i,1)$. 

Moreover any element of order $p$ in $N$ is of the form $(\alpha^i,0)$ and since $C_G(n) = \langle a \rangle$ for a non-trivial $n \in N_q$ we can argue as above to see that $(1,0)$, ..., $(\alpha^{d-1},0)$ are the $C_G(n)$-conjugacy classes of elements of order $p$ in $N$.

\item[\eqref{prop all pq conjugate, G trans auf pq}] Let $(\alpha^i, \beta^j)$ be some element of order $p\cdot q$ in $N$. We will show that there are $r,s,t \in \mathbb{Z}$ such that $(\alpha^i, \beta^j)^{a^rb^sc^t} \in \langle (1,1) \rangle = \mathbb{F}_p \times \mathbb{F}_q$. This is the case if and only if
$$ i + dr + t \equiv 0 \bmod (p+1), \ j + ds + t \equiv 0 \bmod (q+1).$$
These congruences can be solved for any given $i$ and $j$ since $d$ is by assumption coprime to $p+1$ and $q+1$, so we just need to bring $i$, $j$ and $t$ over to the right hand side, and then divide by $d$.

\item[\eqref{prop all pq conjugate, G/C semireg}] $G/C_G(N_p)$ acts on $N_p = \mathbb{F}_{p^2} \times \{0\}$ by multiplication by elements of $\mathbb{F}_{p^2}^\times$. There are $p+1$ cyclic subgroups in $N_p$ each of which has $p-1$ non-trivial cosets. So as $|G/C_G(N_p)| = p^2-1$ it is enough to show that it acts semiregularly. Let $K$ be a cyclic group of order $p$ in $N_p$ and let $m_0 \in N_p \setminus K$. Then we can write $m_0\cdot K$ as a subset of $N_p$ as $(\alpha^i + \alpha^j\cdot \mathbb{F}_p, 0)$ for certain $i$ and $j$. We can understand this as an affine line in the $\mathbb{F}_p$-vector space $\mathbb{F}_{p^2}$. If multiplication by an element $\alpha^r$ stabilises this coset it also stabilises $\alpha^j\cdot \mathbb{F}_p$, which means $\alpha^r \in \mathbb{F}_p$. Hence we get 
$$\alpha^i + \alpha^j\cdot \mathbb{F}_p = \alpha^r \cdot \alpha^i + \alpha^j\cdot \mathbb{F}_p \Leftrightarrow \alpha^i\cdot (1-\alpha^r) \in \alpha^j \cdot \mathbb{F}_p \Leftrightarrow \alpha^i \in \alpha^j\cdot \mathbb{F}_p, $$
contradicting the assumption that $m_0\cdot K$ is a non-trivial coset.

\item[\eqref{prop all pq conjugate, C_A(N_q) semireg}] Since $C_A(N_q) = \langle a \rangle$, and $a$ acts on $N_p$ by multiplication by an element of order $\frac{p^2-1}{d}$ in $\mathbb{F}_{p^2}^\times$ we can argue as in the proof of \eqref{prop all pq conjugate, G/C semireg}.

\item[\eqref{prop all pq conjugate, C_G(N_q) trans}] Clearly $\langle \alpha\rangle = \F_{p^2}^\times$ acts transitively on the set of cyclic groups of order $p$ in $N_p$, since it acts transitively on the set of non-trivial elements.  Multiplying by an element in $\F_p^\times=\langle\alpha^{p+1}\rangle$ stabilises any subgroup of order $p$, since they are of the form $\alpha^i\cdot \mathbb{F}_p$, for some $i$. 
We have $C_G(N_q)=\langle a \rangle$, and $a$ acts by multiplication by $\alpha^d$ on $\F_{p^2}$.
To show that $C_G(N_q)$ acts transitively on subgroups of order $p$ it suffices to show that $\alpha^d$ together with
$\alpha^{p+1}$ generates all of $\F_{p^2}^\times$, since $\alpha^{p+1}$ acts trivially on the set of subgroups of order $p$ of $N_p$ anyway. 
 Since $\gcd(d,p+1)=1$ by assumption, we have $\langle \alpha^d, \alpha^{p+1}\rangle = \langle \alpha \rangle =\F_{p^2}^\times$, which completes the proof.
\end{enumerate}
\end{proof}

We proceed to describe the irreducible complex characters of $G(p,q;d;\alpha,\beta)$. We do this the elementary way, but it could also be done using, for instance, the theory of strong Shoda pairs, cf. \cite[Section 3.5]{GRG1}.

\begin{prop}\label{prop irr chars}
	Let $G=G(p,q;d;\alpha,\beta)$.
	Fix an arbitrary irreducible complex character $\varphi_p\in\Irr_{\C}(N)$   with kernel $\langle (1,0), (0,y)\ |\  y\in \F_{q^2} \rangle$ and an arbitrary irreducible complex character $\varphi_q\in\Irr_{\C}(N)$   with kernel $\langle (x,0), (0,1)\ |\  x\in \F_{p^2} \rangle$. 
	Then the irreducible complex characters of $G$ are given as follows:
	\begin{enumerate}
		\item The characters induced from  the linear characters of $N$ which have kernel $\langle (1,1) \rangle$.
		\item The characters induced from linear characters of $C_G(N_p)$ whose restriction to $N$ is $\varphi_p$.
		\item The characters induced from linear characters of $C_G(N_q)$ whose restriction to $N$ is $\varphi_q$.
		\item The linear characters of $G$. The kernels of these always contains $N$.
	\end{enumerate}
	In particular, each irreducible character of $G$ is induced from a linear character of a subgroup of $G$, and the kernels of these linear characters always contain $\langle (1,1)\rangle$.
\end{prop}
\begin{proof}
	If $\psi$ is a linear characters of $N$ with kernel of order $p\cdot q$, then 
	\begin{equation}\label{eqn psii psij}
		(\psi\uparrow_N^G, \psi\uparrow_N^G) = (\psi, \psi\uparrow_N^G|_N) = \sum_{x\in A} (\psi, \psi^x)
	\end{equation}
	by the Mackey formula. The character $\psi^x$ is again an irreducible character of $N$, and $\psi^x = \psi$ if and only if $\Ker(\psi)^x=\Ker(\psi)$ and $n\cdot \Ker(\psi)=x^{-1}nx\cdot \Ker(\psi)$ for all $n\in N$.
	This only happens if $n\cdot \Ker(\psi)_p = (n\cdot \Ker(\psi)_p)^x$ for every $n\in N_p$ and $n\cdot \Ker(\psi)_q = (n\cdot \Ker(\psi)_q)^x$ for every $n\in N_q$. By the regularity assertions of Proposition~\ref{prop all pq conjugate}~\eqref{prop all pq conjugate, G/C semireg} this implies that
	$x\in C_A(N_p)\cap C_A(N_q) = \{1\}$.  
	 Looking again at the right hand side of \eqref{eqn psii psij} we conclude that $(\psi\uparrow_N^G, \psi\uparrow_N^G)=1$, that is, $\psi\uparrow_N^G$ is irreducible.
	 
	If $\psi'$ is another irreducible character of $N$ with kernel of order $p\cdot q$, then  $(\psi\uparrow_{N}^G,\psi'\uparrow_{N}^G)$ is either zero or one, and, to be more precise, it follows from  Mackey's formula that  $(\psi\uparrow_{N}^G,\psi'\uparrow_{N}^G)=1$ if and only if $\psi'=\psi^x$ for some $x\in A$.
	Hence the number of irreducible characters of $G$ we have constructed so far is the number of $G$-orbits of characters of $N$ with kernel of order $p\cdot q$, and all of these characters have degree $[G:N]=|A|$. 
	
	Now we will show that the number of $G$-orbits of characters of $N$ with kernel of order $p\cdot q$ is equal to $d$. Denote by $\zeta_{p\cdot q}$ a primitive $p\cdot q$-th root of unity. A character $\psi$ of $N$ with kernel of order $p\cdot q$ is uniquely determined by the fibre $\psi^{-1}(\{\zeta_{p\cdot q}\})$, which can be written as a coset $n_\psi \cdot \Ker(\psi)$ for some $n_\psi \in N$ of order $p\cdot q$ such that $(n_\psi)_p \not\in \Ker(\psi)_p$ and $(n_\psi)_q\not\in\Ker(\psi)_q$. The coset $n_\psi \cdot \Ker(\psi)$ is equal to the product of the coset $(n_\psi)_p \cdot \Ker(\psi)_p$ and $(n_\psi)_q \cdot \Ker(\psi)_q$. Conversely, a pair of cosets $n_1X_1$ and $n_2X_1$, with $X_1\leq N$ of order $p$, $n_1\in N_p$ not contained in $X_1$, $X_2\leq N_q$ of order $q$ and $n_2\in N_q$ not contained in $X_2$ determines a character $\psi$ of $N$ with kernel of order $p \cdot q$. The group $G/C_G(N_p)\times G/C_G(N_q)$ acts regularly on the set of such pairs by Proposition~\ref{prop all pq conjugate}~\eqref{prop all pq conjugate, G/C semireg}, and $G/N$ embeds diagonally into $G/C_G(N_p)\times G/C_G(N_q)$ since $C_G(N_p)\cap C_G(N_q)=N$. The index of the image of this embedding is $[A:C_A(N_p)]\cdot [A:C_A(N_q)]\cdot |A|^{-1}=d$. Therefore $G$ has exactly $d$ orbits on such pairs of cosets, and therefore also on characters of $N$ with kernel of order $p\cdot q$.
	
	Next let us show that if $\psi$ is a linear character of $C_G(N_p)$ whose restriction to $N$ is $\varphi_p$ then
	$\psi\uparrow_{C_G(N_p)}^G$ is irreducible. By Frobenius reciprocity and Mackey we have
	\begin{equation}
		(\psi\uparrow_{C_G(N_p)}^G, \psi\uparrow_{C_G(N_p)}^G)= \sum_{x\in A/C_A(N_p)} (\psi, \psi^x)
	\end{equation}
	Denote by $\zeta_p$ a primitive $p$-th root of unity. We have $\psi^{-1}(\{\zeta_p\})\cap N_p = \varphi_p^{-1}(\{\zeta_p\})$, which is a coset of $\Ker(\varphi_p)$, a group of order $p$. Since 
	$A/C_A(N_p)$ acts regularly on the set of non-trivial cosets of subgroups of order $p$ in $N_p$ by Proposition~\ref{prop all pq conjugate}~\eqref{prop all pq conjugate, G/C semireg} it follows that	$(\psi^x)^{-1}(\{\zeta_p\})\cap N_p=(\psi^{-1}(\{\zeta_p\})\cap N_p)^x\neq \psi^{-1}(\{\zeta_p\})\cap N_p$ (and therefore $\psi \neq \psi^x$) whenever $x\in A$ such that $x \notin C_A(N_p)$. 
	The irreducibility of $\psi\uparrow_{C_G(N_p)}^G$ now follows.
	Moreover, if we have another linear character $\psi'$ of $C_G(N_p)$ such that $\psi'|_{N}=\varphi_p$ and 
	$\psi'\neq \psi$, then $\psi(x)\neq\psi'(x)$ for some $x\in C_A(N_p)$. But $\psi\uparrow_{C_G(N_p)}^G|_A = \psi\uparrow_{C_A(N_p)}^A$, which restricted to $C_A(N_p)$ is just $[A:C_A(N_p)]\cdot \psi$. Since the same holds for $\psi'$ it follows that $\psi\uparrow_{C_G(N_p)}^G(x) \neq \psi'\uparrow_{C_G(N_p)}^G(x)$. We have 
	$|C_G(N_p)/N_p|=|C_A(N_p)|$ possibilities for $\psi$ in total, and the degree of $\psi\uparrow_{C_G(N_p)}^G$ is 
	$[G:C_G(N_p)]=[A:C_A(N_p)]$.
	
	Analogously we get $|C_A(N_q)|$ characters of the form given in the third point of the statement,
	and their degrees are $[A:C_A(N_q)]$. 
	As for the linear characters, there are $|G/N|=|A|$ irreducible characters with $N$ in their kernel.
	
	These four families of irreducible characters of $G$ are disjoint owing to the fact that the intersection of the kernel of a character with $N$ is something different depending on the family the character comes from (it is either $\{1\}$, $N_q$, $N_p$ or $N$). So all that is left to do now is check that the sum of the squares of the degrees of the characters we have constructed is equal to $|G|$: 
	$$
		\begin{array}{rl}
		&d\cdot |A|^2 + |C_A(N_p)|\cdot [A:C_A(N_p)]^2 + |C_A(N_q)|\cdot [A:C_A(N_q)]^2 + |A|\cdot 1^2 \\\\
		=& |A|\cdot (d\cdot |A| + [A:C_A(N_p)]+[A:C_A(N_q)]+ 1) \\\\
		=& |A|\cdot ((p^2-1)\cdot (q^2-1) + (p^2-1) + (q^2-1) + 1) \\\\
		=& |A| \cdot p^2\cdot q^2=|G|
		\end{array}
	$$
\end{proof}

\begin{prop}\label{prop eigenvalue 1}
	Let $G=G(p,q;d;\alpha,\beta)$ and let $U=\langle c \rangle$ be a cyclic group of order $p\cdot q$. Let
	$\eps\OfU: G\longrightarrow \Z:\ g \mapsto \eps_{g^G}\OfU$ be a class function with 
	$
		\sum_{g^G} \eps_{g^G}\OfU = 1
	$ and $\eps_{g^G}\OfU\neq 0$ only for elements $g\in G$ of order $p\cdot q$. Define
	$$
		\chi = \sum_{g^G} \eps_{g^G}\OfU\cdot 1\uparrow_{[g]}^{G\times U}
	$$
	Then
	$$
		(\chi, \eta\otimes 1_U) \neq 0 \quad\textrm{ for all $\eta\in \Irr_{\C}(G)$}
	$$
\end{prop}
\begin{proof}
	By Proposition \ref{prop all pq conjugate}~\eqref{prop all pq conjugate, G trans auf pq} we can choose  
	$g_1,\ldots, g_k\in \langle (1,1) \rangle$ such that $g_i^G\neq g_j^G$ whenever $i\neq j$ and 
	$\eps_{g^G}\OfU\neq 0$ for some $g\in G$ if and only if $g^G=g_i^G$ for some $i$. Our assumptions ensure that each $g_i$ has order $p\cdot q$, which implies that it generates $\langle (1,1)\rangle$. Now, if $\eta\in\Irr_{\C}(G)$ then
	$$
		(\chi, \eta \otimes 1_U) 
		=\sum_{i=1}^k \eps_{g_i^G}\OfU\cdot \left(1\uparrow_{[g_i]}^{G\times U}, \eta\otimes 1_U \right)
	$$
	We claim that the value of $\left(1\uparrow_{[g_i]}^{G\times U}, \eta\otimes 1_U\right)$ is independent of $i$, which would imply that $(\chi, \eta \otimes 1_U) = (1\uparrow_{[g_1]}^{G\times U}, \eta \otimes 1_U)$ (of course we could have used any $g_i$ here instead of $g_1$). By Frobenius reciprocity we have 
	$$
			\left(1\uparrow_{[g_i]}^{G\times U},\eta\otimes1_U\right)=
			\left(1_{[g_i]},(\eta\otimes1_U)|_{[g_i]}\right)=
			\displaystyle \frac{1}{p \cdot q}\cdot \sum_{j=1}^{p \cdot q} \eta(g_i^j)
	$$
	The value of the right hand side manifestly only depends on the group generated by $g_i$, which is $\langle (1,1)\rangle$ independent of the value of $i$.
	
	It remains to be seen that $(1\uparrow_{[g_1]}^{G\times U}, \eta \otimes 1_U)$ is non-zero for every $\eta$. By Proposition \ref{prop irr chars} any $\eta\in \Irr_{\C}(G)$ can be written as 
	$\varphi\uparrow_K^G$ where $\langle (1,1)\rangle \leq K\leq G$ and $\varphi$ is a linear character of $K$ whose kernel contains $\langle (1,1)\rangle =\langle g_1 \rangle$. Hence
	$$
		\begin{array}{rcl}
		(1\uparrow_{[g_1]}^{G\times U}, \eta \otimes 1_U) &=&
		(1\uparrow_{[g_1]}^{G\times U}, (\varphi \otimes 1_U)\uparrow_{K\times U}^{G\times U}) 
		= 	(1_{[g_1]}, (\varphi \otimes 1_U)\uparrow_{K\times U}^{G\times U}|_{[g_1]}) \\ \\
		&=& \displaystyle \sum_{x\in K\times U \setminus G\times U/[g_1]} \left(1_{[g_1]}, (\varphi\otimes 1_U)\uparrow_{(K\times U)^x\cap [g_1]}^{[g_1]}\right)\\\\ &=& \left(1_{[g_1]}, (\varphi\otimes 1_U)|_{[g_1]} \right)  + \textrm{ (other terms)} = (1_{[g_1]}, 1_{[g_1]}) + \textrm{ (other terms)}
		\end{array}
	$$
	which is clearly greater than zero.
\end{proof}

\begin{lemma}\label{lemma xi proper} Let $G = G(p,q;d;\alpha,\beta)$, let $\varepsilon\OfU: G \rightarrow \mathbb{Z}$ be a class function which is non-vanishing only on elements of $N$ of order $p\cdot q$ such that 
$$
	\sum_{g^G} \eps_{g^G}\OfU = \sum_{i=1}^d \eps_{(\alpha^i,1)^G}= 1
$$	
and let $U=\langle c \rangle$ be a cyclic group of order $p\cdot q$. Set $n = (0,1)\in N_q$. Recall also the definition of $K_p$ and the $r_i(p)$'s from Notation \ref{notation K_p, r_i}. We will write $r_i$ instead of $r_i(p)$ below.
	\begin{enumerate}
		\item The character
		$$
			\xi_n = \sum_{m \in N_{p}} \varepsilon_{(m\cdot n)^G}\OfU\cdot 1\uparrow_{[m]_{p}}^{N_{p}\times U_{p}}  \quad \textrm{ (same as in Remark \ref{remark formula xin})}
		$$
		is a proper character of $N_p \times U_p$ %for every $n \in N_q \setminus \{(0,0)\}$
		if and only if
		$$\begin{pmatrix} r_1 & r_2 & \dots & r_d \\ r_2 & r_3 & \dots & r_1 \\ \vdots & \vdots & \ddots & \vdots \\ r_d & r_1 & \dots & r_{d-1} \end{pmatrix} \begin{pmatrix} \varepsilon_{(\alpha,1)^G}\OfU \\ \varepsilon_{(\alpha^2,1)^G}\OfU \\ \vdots \\ \varepsilon_{(\alpha^d,1)^G}\OfU \end{pmatrix} \geq 0.$$
		
		\item Let $\varphi$ be an irreducible rational character of $N_p \times U_p$. Let 
		$$K = \Ker(\varphi)\cap N_p$$
		Then the values of 
		$$\mu(\varphi, n) = \frac{(\varphi, \xi_{n})}{(\varphi,\varphi)} \cdot \frac{1}{[C_G(n)\cap N_G({\rm Ker}(\varphi)):N]} \quad \textrm{(same as in Lemma \ref{lemma extending chin})}$$
		are as follows:
		\begin{enumerate}
			\item If $\varphi$ is the trivial character then $\mu(\varphi,n) = 1$.
			\item If there is no element $m_0 \in N_p$ such that $(m_0,c_p) \in \Ker(\varphi)$ then $\mu(\varphi, n) = 0$.
			\item Otherwise choose an $m_0 \in N_p$ such that $(m_0,c_p)\in \Ker(\varphi)$. 
			\begin{enumerate}
			\item If $m_0 \in K$ then $\mu(\varphi,n) = 1$. 
			\item If $m_0\not\in K$ we can choose a $g \in G$ such that $(m_0\cdot K)^g = (\alpha + \mathbb{F}_p,0)$ by Proposition~\ref{prop all pq conjugate}~\eqref{prop all pq conjugate, G/C semireg}. Choose  $\ell(g)\in \Z$ such that $(1,0)^g = (\alpha^{\ell(g)},0)$. Then we have
			$$ \mu(\varphi,n) = \sum_{i=0}^{d-1} r_{\ell(g)+i}\cdot \varepsilon_{(\alpha^i,1)^G}\OfU  $$
			\end{enumerate}
		\end{enumerate} 
	\end{enumerate}
	If we reverse the roles of $p$ and $q$ as well as $\alpha$ and $\beta$ then the same statements also hold for $\xi_{n}$ with $n = (1,0) \in N_p$.
\end{lemma}
\begin{proof}
	%First note that because the matrix of inequalities in our assumption is circular and because $(\alpha^s,1)$ and $(\alpha^t,1)$ are conjugate in $G$ if and only if $s \equiv t \bmod d$ for any fixed $j$ we have
	%$$\begin{pmatrix} r_1 & r_2 & \dots & r_d \\ r_2 & r_3 & \dots & r_1 \\ \vdots & \vdots & \vdots & \vdots \\ r_d & r_1 & \dots & r_{d-1} \end{pmatrix} \begin{pmatrix} \varepsilon_{(\alpha^{j+1},1)}\OfU \\ \varepsilon_{(\alpha^{j+2},1)}\OfU \\ \vdots \\ \varepsilon_{(\alpha^{j+d},1)}\OfU \end{pmatrix} \geq 0.$$
	
	%Moreover 
	For every $j\in\{0,\ldots,d-1\}$ the inequalities holding by assumption can be understood as follows:
	\begin{align}\label{eqn ineq assumed} 0 \leq \sum_{i=1}^d r_{j+i}\cdot \varepsilon_{(\alpha^i,1)^G} &= \sum_{i=1}^d |\{(\alpha^{j+i},0)^{\langle a \rangle} \cap K_p \}|\cdot  \varepsilon_{(\alpha^i,1)^G}\OfU \\
	&= \sum_{i=1}^d |\{(\alpha^{j+i},0)^{\langle a \rangle} \cap (\alpha + \mathbb{F}_p,0) \}|\cdot  \varepsilon_{(\alpha^i,1)^G}\OfU
	\end{align}
	
	Furthermore, for any fixed $i$ we have
	\begin{align}\label{p-1Durchd}
	\left|(\alpha^i,0)^{\langle a \rangle} \cap (\mathbb{F}_p,0)\right| &= \left|\left\{1\leq t \leq \frac{p^2-1}{d} \ \bigg| \ \alpha^{i+td} \in \mathbb{F}_p \right\}\right| \\ 
	&= \left|\left\{1\leq t \leq \frac{p^2-1}{d} \ \bigg| \ i+td \equiv 0 \bmod (p+1) \right\}\right| = \frac{p-1}{d}
	\end{align}
	where the last equality follows since $d$ and $p+1$ are coprime.
	
	%Let $n = (0,\beta^s) \in N_q \setminus \{(0,0)\}$. Then 
	Now $C_G(n) = \langle a \rangle$ and representatives of the $C_G(n)$-conjugacy classes in $N_p \setminus \{(0,0)\}$ are given by $(\alpha,0),(\alpha^2,0),...,(\alpha^d,0)$ by Proposition~\ref{prop all pq conjugate}~\eqref{prop all pq conjugate, K'klassen}. 
	So by Lemma~\ref{ScalarProdXiPsi} we know that $\xi_{n}$ is a proper character of $N_p \times U_p$ if and only if for every subgroup $K = (\mathbb{F}_p\cdot \alpha^{s},0)$ of order $p$ in $N_p$, where $s\in \{1,\ldots,p^2-1 \}$, and every $(m_0,0) \in N_p$ we have 
	$$\sum_{i=0}^{d-1} \left|\left\{(\alpha^i, 0)^{\langle a \rangle} \cap (m_0+\F_p\cdot \alpha^s,0)\right\}\right| \cdot \varepsilon_{(\alpha^i, 1)^G}\OfU \geq 0. $$
	If $m_0 \notin \mathbb{F}_p\cdot \alpha^{s}$ then by Proposition~\ref{prop all pq conjugate}~\eqref{prop all pq conjugate, G/C semireg} there exists a $g \in G$ such that $(m_0+\F_p\cdot \alpha^s,0)^g = (\alpha + \mathbb{F}_p,0)$. Pick an $\ell(g)$ such that $(1,0)^g = (\alpha^{\ell(g)},0)$.
	%Also note that $(\alpha^i,\beta^s)$ is conjugate in $G$ by $c^{-s}$ to $(\alpha^{i-s}, 1)$. 
	So the condition we have to verify can be formulated as 
	$$\sum_{i=0}^{d-1} \left|\left\{(\alpha^{i+\ell(g)}, 0)^{\langle a \rangle} \cap (\alpha+\mathbb{F}_p,0)\right\}\right|\cdot  \varepsilon_{(\alpha^i, 1)^G}\OfU \geq 0 $$
	and this holds by \eqref{eqn ineq assumed} with $j=\ell(g)$.
	
	So assume $m_0 \in  \mathbb{F}_p\cdot \alpha^{s}$ and let $g \in C_G(n)$ be chosen such that $ (\mathbb{F}_p\cdot \alpha^{s},0)^g = (\mathbb{F}_p, 0)$, which exists by Proposition~\ref{prop all pq conjugate}~\eqref{prop all pq conjugate, C_G(N_q) trans}. Then by \eqref{p-1Durchd} we have
	\begin{align*}
	\sum_{i=0}^{d-1} & \left|\left\{(\alpha^i, 0)^{\langle a \rangle} \cap (m_0+\F_p\cdot \alpha^s,0)\right\}\right|\cdot \varepsilon_{(\alpha^i, 1)^G}\OfU = \sum_{i=0}^{d-1} \left|\left\{(\alpha^i, 0)^{g\langle a \rangle} \cap (\F_p\cdot \alpha^s, 0)^g\right\}\right|\cdot  \varepsilon_{(\alpha^i, 1)^G}\OfU \\
	= & \sum_{i=0}^{d-1} \left|\left\{(\alpha^{i + \ell(g)}, 0)^{\langle a \rangle} \cap (\mathbb{F}_p,0) \right\}\right|\cdot \varepsilon_{(\alpha^i, 1)^G}\OfU = \sum_{i=0}^{d-1} \frac{p-1}{d}\cdot \varepsilon_{(\alpha^i, 1)^G}\OfU = \frac{p-1}{d}
	\end{align*}
	This finishes the proof of the first claim.
	
	Let $\varphi_0$ be an irreducible complex character of $N_p \times U_p$ such that $\varphi$ is the sum of the Galois-conjugates of $\varphi_0$. We can reformulate the definition of $\mu(\varphi, n)$ as
	\begin{align}\label{MuFormel}
	\mu(\varphi, n) = \frac{(\varphi_0, \xi_n)}{[C_G(n)\cap N_G(\Ker(\varphi)):N]} 
	\end{align}
	If $\varphi$ is the trivial character then $C_G(n) \cap N_G(\Ker(\varphi)) = C_G(n)$. Using Lemma \ref{ScalarProdXiPsi} we get
	$$(\varphi_0, \xi_n) = [C_G(n):N] \cdot \sum_{(n\cdot m)^G, m \in N_{p}} \eps_{(n\cdot m)^G}\OfU=[C_G(n):N]$$
	which shows that $\mu(\varphi,n)$ is as desired.
	 
	Next, if there is no element $m_0 \in N_p$ such that $(m_0,c_p)\in \Ker(\varphi)$ then Lemma \ref{ScalarProdXiPsi} implies that $(\varphi_0,\xi_n)=0$, and again $\mu(\varphi,n)$ is as desired.
	
	So let us assume that we have an $m_0 \in N_p$ such that $(m_0,c_p) \in \Ker(\varphi)$.
	Then by Lemma \ref{ScalarProdXiPsi}
	$$
		(\varphi_0, \xi_n)= \sum_{m^{C_G(n)},m\in N_p} \left|m^{C_G(n)} \cap m_0\cdot K\right| \cdot \varepsilon_{(n\cdot m)^G}\OfU
	$$
	 If $m_0 \in K$ then $\Ker(\varphi)=K\times U$, and therefore $N_G(\Ker(\varphi))=N_G(K)$. Since $C_G(n)$ acts transitively on the set of cyclic groups of order $p$ in $N_p$ (by Proposition~\ref{prop all pq conjugate}~\eqref{prop all pq conjugate, C_G(N_q) trans}), and $N_G(\Ker(\varphi))$ is the stabiliser in $G$ of one of these cyclic groups of order $p$ (namely $K$), it follows that $[C_G(n):C_G(n) \cap N_G(\Ker(\varphi))]$ is equal to the number of cyclic subgroups of order $p$ of $N_p$, which is $p+1$. By the regularity asserted in Proposition~\ref{prop all pq conjugate}~\eqref{prop all pq conjugate, G/C semireg} $[G:C_G(n)]=[G:C_G(N_q)]=q^2-1$. Therefore $$
	 	\begin{array}{rcl}
	 	[C_G(n)\cap N_G(\Ker(\varphi)):N] &=& \displaystyle \frac{[G:N]}{[G:C_G(n)]\cdot [C_G(n):C_G(n)\cap \Ker(\varphi)]}\\\\ &=&\displaystyle
	 	\frac{(p^2-1)\cdot (q^2-1)}{d} \cdot \frac{1}{q^2-1}\cdot \frac{1}{p+1} = \frac{p-1}{d}
	 	\end{array}
	 $$ 
	 By Proposition \ref{prop all pq conjugate}~\eqref{prop all pq conjugate, C_G(N_q) trans} we can also find a $g\in G$ such that $K^g=(\F_p,0)$.
	 So 
	 $$
	 	\begin{array}{rcl}
	 	\mu(\varphi, n) &=&\displaystyle \frac{d}{p-1} \cdot \sum_{m^{C_G(n)}, m\in N_p} \left|m^{C_G(n)} \cap K\right| \cdot \varepsilon_{(n\cdot m)^G}\OfU\\\\ &=&\displaystyle \frac{d}{p-1}\cdot \sum_{i=0}^{d-1} \left| (\alpha^i,0)^{\langle a \rangle} \cap (\F_p,0)\right| \cdot \varepsilon_{(\alpha^i, 1)^G}\OfU = 
	 	\displaystyle \frac{d}{p-1}\cdot \sum_{i=0}^{d-1} \frac{p-1}{d}\cdot \varepsilon_{(\alpha^i, 1)^G}\OfU=1
	 	\end{array}
	 $$
	 where we used \eqref{p-1Durchd} to compute the cardinalities. This settles the case $m_0\in K$.
	
	Finally assume $m_0 \notin K$. Since $C_G(n)$ acts semiregularly on the non-trivial cosets of cyclic groups of order $p$ in $N_p$ and $N_G(\Ker(\varphi))$ fixes the coset $m_0\cdot K$ we have $C_G(n) \cap N_G(\Ker(\varphi)) = N$. It follows that $\mu(\varphi, n)=(\varphi_0, \xi_n)$. We may again choose a $g\in G$ such that $(m_0\cdot K)^g=(\alpha+\F_p,0)$, and we can define $\ell(g)$ as before. Hence
	\begin{align*}
	 (\varphi_0, \xi_n)&= \sum_{m^{C_G(n)}, m \in N_p} \left|m^{C_G(n)} \cap (m_0\cdot K)\right|\cdot  \varepsilon_{(m\cdot n)^G}\OfU   \\
	&= \sum_{i=0}^{d-1} \left|(\alpha^i,0)^{\langle a \rangle} \cap (m_0\cdot K)\right| \cdot \varepsilon_{(\alpha^i, 1)^G}\OfU = \sum_{i=0}^{d-1} \left|(\alpha^{i+\ell(g)},0)^{\langle a \rangle} \cap (\alpha + \mathbb{F}_p,0)\right|\cdot \varepsilon_{(\alpha^i, 1)^G}\OfU \\
	&= r_{\ell(g)}\cdot \varepsilon_{(1,1)^G}\OfU + r_{\ell(g)+1}\cdot \varepsilon_{(\alpha,1)^G}\OfU + \ldots + r_{\ell(g)+d-1}\cdot \varepsilon_{(\alpha^{d-1},1)^G}\OfU
	\end{align*}
	as claimed.
\end{proof}

\begin{proof}[Proof of Theorem \ref{thm zc Gpqd}]
	First we need to check that $\eps\OfU$ satisfies the conditions of Theorem \ref{thm existence u semilocal}, that is, 
	\ref{cond sum 1}-\ref{cond xi proper char}.
	\begin{enumerate}
		\item Condition \ref{cond sum 1} is satisfied by definition of $\eps\OfU$. 
		\item Condition \ref{cond cent n eq N} is satisfied by Proposition \ref{prop all pq conjugate}~\eqref{prop all pq conjugate, N selbstzentrali}. 
		\item Condition \ref{cond xi proper char} holds by Lemma \ref{lemma xi proper}.
	\end{enumerate}
	This furnishes us with a unit of order $p\cdot q$ in $\UU(\Z_{\pi}G)$ having the desired partial augmentations. Now let us check the conditions of Theorem \ref{thm existence u global}, that is, \ref{cond eichler} and \ref{cond eigenvalue}.
	\begin{enumerate}
		\item We need to show that $G$ does not have an epimorphic image isomorphic to one of the groups in the list given in Theorem~\ref{thm existence u global}.
		Since all groups in that list are non-commutative subgroups of the real quaternions they all have an irreducible complex character of degree two. It follows from Proposition \ref{prop irr chars} that the degrees of the irreducible characters of $G$ are 
			$$\left\{1,\ p^2-1,\ q^2-1, \frac{(p^2-1)\cdot (q^2-1)}{d}\right\}$$
		All of these numbers, except for $1$, are greater than or equal to $2^2-1=3$, so clearly $G$ cannot surject onto a group which has an irreducible complex character of degree two.
		\item We need to check that $(\chi, \varphi\otimes 1_U)\neq 0$ for all $\varphi \in \Irr_{\C}(G)$, where $U$ is  cyclic group of order $p\cdot q$  and $\chi$ is obtained from $\eps\OfU$ via formula \eqref{eqn def chi}. This follows by Proposition \ref{prop eigenvalue 1}.
	\end{enumerate}
	This yields a unit $u\in \UU(\Z G)$, which also has partial augmentations given by $\eps\OfU$. It follows immediately from the double action formalism (Propositions \ref{prop double action iso implies conj} and \ref{prop character double action module}) that if $\eps\OfU$ is non-vanishing on more than one conjugacy class, then $u$ is not conjugate in $\UU(\Q G)$ to an element of the form $\pm g$ for $g\in G$.
\end{proof}

\begin{proof}[Proof of Theorem A]
Set $G = G(7,19;3;\alpha,\beta)$, where $\alpha$ is a root of the polynomial $X^2-X+3$ over $\mathbb{F}_7$ and $\beta$ is a root of $X^2-X+2$ over $\mathbb{F}_{19}$.  Let $U = \langle c \rangle$ be a cyclic group of order $7 \cdot 19$.   Note that representatives of the conjugacy classes of elements of order $7 \cdot 19$ in $G$ are given by $(1,1)$, $(1,\beta)$, $(1,\beta^2)$, or, alternatively, $(1,1)$, $(\alpha, 1)$, $(\alpha^2, 1)$. 
We will need to use both systems of representatives, and to avoid confusion we should note that  $(1, \beta)$ is conjugate to $(\alpha^2, 1)$ since $(1, \beta)^{c^{-1}\cdot a} = (\alpha^{-1}, 1)^a = (\alpha^2, 1)$.

Define a class function $\varepsilon\OfU: G \rightarrow \mathbb{Z}$ vanishing everywhere except on the conjugacy classes of $(1,1)$ and $(1,\beta^2)$. Let the values of $\eps\OfU$ on these two classes be given by
$$\varepsilon_{(1,1)^G}\OfU = 2 \quad\textrm{and}\quad \varepsilon_{(1,\beta^2)^G}\OfU = -1$$

All we have to do now is check that $\eps\OfU$ satisfies the inequalities \eqref{eqn inequality rp} and \eqref{eqn inequality rq}. Theorem \ref{thm zc Gpqd} then shows that this $G$ does indeed constitute a counterexample to the Zassenhaus conjecture. 

First assume that $n=(0,1) \in N_{19}$.
As in Notation~\ref{notation K_p, r_i} set $K_7 = (\alpha + \F_7,0)$. For $1 \leq i \leq 3$ define
$$
	A_i = \left\{1 \leq t \leq \frac{p^2-1}{d} \ \bigg| \ (\alpha^{i+3\cdot t},0) \in K_7 \right\}
$$ 
That is, $|A_i| = r_i(7)$. Denote by $\Nr$ the usual Galois norm of $\mathbb{F}_{7^2}$ over $\mathbb{F}_7$. Then for $x \in \mathbb{F}_7$ we have
$$\Nr(\alpha +x) = x^2+(\alpha + \alpha^p)x + \alpha^{p+1} = x^2+x+3 $$
where the last equality follows from the fact that $\alpha^p+\alpha$ is the trace of $\alpha$ over $\mathbb{F}_7$ and $\alpha^{p+1}$ its norm. Those are just the coefficients occurring in the minimal polynomial of $\alpha$, where the trace is taken negatively. We have $\Nr(\alpha^{i+3\cdot t}) = (\alpha^{p+1})^{i+3\cdot t} = 3^{i+3\cdot t}= (-1)^t\cdot 3^i$. So $\alpha + x \in A_i$ if and only if $\Nr(\alpha + x) \in \{\pm 3^i \}$. Computing these norms for every $x \in \mathbb{F}_7$ we get the values in Table~\ref{table F7}.
\begin{table}[h]
	\caption{Computation of $A_i$ for $\F_7$}
	\begin{tabular}{lccccccc}
		$x \in \mathbb{F}_7$: & $0$ & $1$ & $2$ & $3$ & $-3$ & $-2$ & $-1$  \\ \hline 
		$\Nr(\alpha + x)$: & $3$ & $-2$ & $2$ & $1$ & $2$ & $-2$ & $3$ \\ 
		$i$ such that $\alpha + x \in A_i$: & $1$ & $2$ & $2$ & $3$ & $2$ & $2$ & $1$
	\end{tabular}\label{table F7}
\end{table}
We conclude $r_1(7) = 2$, $r_2(7) = 4$ and $r_3(7) = 1$. The inequalities \eqref{eqn inequality rp} from Theorem \ref{thm zc Gpqd}, written in matrix form, now read as follows
$$\begin{pmatrix} 2 & 4 & 1 \\ 4 & 1 & 2 \\ 1 & 2 & 4 \end{pmatrix}\cdot  \begin{pmatrix} \varepsilon_{(\alpha,1)^G}\OfU \\ \varepsilon_{(\alpha^2,1)^G}\OfU \\ \varepsilon_{(1, 1)^G}\OfU \end{pmatrix} =  \begin{pmatrix} 2 & 4 & 1 \\ 4 & 1 & 2 \\ 1 & 2 & 4 \end{pmatrix}\cdot  \begin{pmatrix} -1 \\ 0 \\ 2 \end{pmatrix} \geq 0 $$
and they clearly hold.

Now let $n=(1,0) \in N_7$. We will argue similarly as above. Let $\Nr$ be the norm of $\mathbb{F}_{19^2}$ over $\mathbb{F}_{19}$. Define subsets $A_1, A_2, A_3 \subseteq K_{19}$ as before. Then for $x \in \mathbb{F}_{19}$ we have $\Nr(\beta + x) = x^2+x+2$. Moreover $\Nr(\beta^{i+3\cdot t}) = 2^{i+3\cdot t}$, so 
$$\Nr(A_i) \subseteq \{8^t\cdot 2^i \ | \ 1\leq t \leq 6 \} = \{2^i, 8 \cdot 2^i, 7 \cdot 2^i, -1 \cdot 2^i, -8 \cdot 2^i, -7 \cdot 2^i \}$$
Hence $$\Nr(A_1) \subseteq \{2, -3,-5,-2,3,5 \}$$ $$\Nr(A_2) \subseteq \{4, -6, 9, -4, 6, -9\}$$ $$\Nr(A_3) \subseteq \{8, 7, -1, -8, -7, 1 \}$$
Computing the norms of elements in $\beta + \mathbb{F}_{19}$ we obtain the values in Table~\ref{table F19}.
\begin{table}[h]
	\caption{Computation of $A_i$ for $\F_{19}$}
	\scalebox{0.9}{
	\begin{tabular}{lccccccccccccccccccc}
		$x \in \mathbb{F}_{19}$: & $0$ & $1$ & $2$ & $3$ & $4$ & $5$ & $6$ & $7$ & $8$ & $9$ & $-9$ & $-8$ & $-7$ & $-6$ & $-5$ & $-4$ & $-3$ & $-2$ & $-1$ \\ \hline
		$\Nr(\alpha + x)$: &  $2$ & $4$ & $8$ & $-5$ & $3$ & $-6$ & $6$ & $1$ & $-2$ & $-3$ & $-2$ & $1$ & $6$ & $-6$ & $3$ & $-5$ & $8$ & $4$  & $2$ \\ 
		$i$ such that $\alpha + x \in A_i$: & $1$ & $2$ & $3$ & $1$ & $1$ & $2$ & $2$ & $3$ & $1$ & $1$ & $1$ & $3$ & $2$ & $2$ & $1$ & $1$ & $3$ & $2$ & $1$
	\end{tabular}}\label{table F19}
\end{table}
So $r_1(19) = 9$, $r_2(19) = 6$ and $r_3(19) = 4$. Hence the inequalities \eqref{eqn inequality rq} from Theorem \ref{thm zc Gpqd}, written in matrix form, are
$$\begin{pmatrix} 9 & 6 & 4 \\ 6 & 4 & 9 \\ 4 & 9 & 6 \end{pmatrix}\cdot \begin{pmatrix} \varepsilon_{(1,\beta)^G}\OfU \\ \varepsilon_{(1,\beta^2)^G}\OfU \\ \varepsilon_{(1, 1)^G}\OfU \end{pmatrix} = \begin{pmatrix} 9 & 6 & 4 \\ 6 & 4 & 9 \\ 4 & 9 & 6 \end{pmatrix}\cdot  \begin{pmatrix} 0 \\ -1 \\ 2 \end{pmatrix} \geq 0 $$
and these also hold. This completes the proof, as all of our assertions now follow from Theorem \ref{thm zc Gpqd}.
\end{proof}

\begin{remark}
%	\begin{enumerate}
%		\item 
Theorems A and B assert the existence of certain units $u\in \UU(\Z G)$ for $G=G(p,q;d;\alpha,\beta)$. To do this we prove the existence of a double action module with the appropriate character.
		Describing the double action module $_u( \Z G)_G$ and the unit $u$ explicitly would be difficult, but in principle even this could be done. However, due to the size of $G$, this might not be feasible in practice and there is no guarantee that the resulting description would be ``nice''.
		
		By contrast, the construction of the $p$-local double action modules $_u( \Z_{(p)} G)_G$ for arbitrary prime numbers $p$ is perfectly explicit. We will do this in Proposition~\ref{prop explicit comp} below for the situation described in Theorem A. 
		The lattices constructed in Proposition~\ref{prop explicit comp} become projective upon restriction to $G$ by definition, and one can use Remark~\ref{remark characters} to compute their characters (on a computer rather than by hand). One can then go on to verify the conditions of Theorems \ref{thm existence u semilocal} and \ref{thm existence u global} directly,
		avoiding the various technical lemmas and cumbersome computations of this section.
%		\item We should note that $G(7,19;3;\alpha,\beta)$ is far from the only counterexample to the Zassenhaus conjecture one can construct using Theorem~\ref{thm zc Gpqd}. For instance, we checked that $G(13,19;3;\alpha,\beta)$, $G(13,31;3;\alpha,\beta)$, and many more groups of the form $G(p,q;3;\alpha,\beta)$ where $p$ and $q$ are primes congruent to one mod $6$ all are counterexamples for certain choices of $\alpha$ and $\beta$. The particular form of the partial augmentations is not important either, and one finds many counterexamples with different distributions of the partial augmentations.    
%	\end{enumerate}
\end{remark}

\begin{lemma}\label{lemma phi reps}
	Let $G = G(p,q;d;\alpha,\beta)$, $U = \langle c \rangle$ a cyclic group of order $p\cdot q$ and $n=(0,1)\in N_q$. Then representatives for the elements of $\Irr_\mathbb{Q}(N_p \times U_p)/{C_G(n)}$ are given by
	\begin{enumerate}
		\item The trivial character
		\item A character whose kernel equals $N_p$
		\item A non-trivial character whose kernel contains $U_p$
		\item For every $0 \leq i \leq d-1$ a non-trivial character $\varphi_i$ whose kernel is 
		$$\langle ((1,0),1), (\alpha^{i\cdot (p+1)} \cdot \alpha, 0),c_p) \rangle$$
	\end{enumerate}
\end{lemma}

\begin{proof}
	A rational character $\varphi$ of $N_p \times U_p$, an elementary-abelian group of rank three, is determined by its kernel. The kernel can be the whole group, which happens if and only if $\varphi$ is the trivial character, or an elementary abelian group of rank two. There are $\frac{p^3-1}{p-1} = p^2+p+1$ cyclic groups of order $p$ in $N_p \times U_p$ and as many subgroups of order $p^2$ by duality. Clearly $N_p$ is a subgroup invariant under the action of $C_G(n) = \langle a \rangle$. Furthermore $C_G(n)$ acts transitively on the cyclic subgroups of $N_p$ by Proposition~\ref{prop all pq conjugate}~\eqref{prop all pq conjugate, C_G(N_q) trans} and therefore also on the groups of the form $\langle m \rangle \times U_p$ with $m$ a non-trivial element of $N_p$. All rational characters which have any of these groups as their kernel are therefore conjugate under the action of $C_G(n)$.
	
	Since the number of cyclic subgroups of $N_p$ is equal to $p+1$ this leaves $p^2+p+1 - 1 - (p+1) = p^2-1$ possible kernels of irreducible characters. Such a kernel is generated by an element $m$ in $N_p$ and an element of the form $(m_0,c)$ with $m_0$ a non-trivial element in $N_p$ such that $m_0 \notin \langle m \rangle$. Hence this kernel is determined by the non-trivial coset $m_0 \cdot \langle m \rangle$. By Proposition~\ref{prop all pq conjugate}~\eqref{prop all pq conjugate, C_A(N_q) semireg} the group $C_A(n)$ acts semiregularly on these cosets, and since $a$ has order $\frac{p^2-1}{d}$ the action of $C_A(n)$ partitions the remaining possible kernels into $d$ orbits.
	
	It remains to show that cosets of the form $(\alpha^{i\cdot (p+1)}\cdot\alpha + \mathbb{F}_p,0)$ and $(\alpha^{j\cdot (p+1)}\cdot\alpha + \mathbb{F}_p,0)$ are not $C_G(n)$-conjugate for any $0\leq i,j \leq d-1$ with $i \neq j$. If $x \in G$ is an element conjugating $(\alpha^{i\cdot (p+1)}\cdot \alpha + \mathbb{F}_p,0)$ into $(\alpha^{j\cdot (p+1)}\cdot\alpha + \mathbb{F}_p,0)$ then $x$ stabilises $(\mathbb{F}_p,0)$ and hence corresponds to multiplication by an element in $\mathbb{F}_p^\times$. The subgroup of $\langle a \rangle$ acting by multiplication by elements of $\F_p^\times$
	is generated by $a^{p+1}$, since $a$ acts by multiplication by $\alpha^d$,  $\F_p^\times = \langle\alpha^{p+1}\rangle$ and $\gcd(d,p+1)=1$ by assumption. But $a^{p+1}$ acts by multiplication by $\alpha^{(p+1)\cdot d}$, so an $x$ lying in the group generated by $a^{p+1}$ could not possibly conjugate $(\alpha^{i\cdot (p+1)}\cdot \alpha + \mathbb{F}_p,0)$ into $(\alpha^{j\cdot (p+1)}\cdot\alpha + \mathbb{F}_p,0)$.
\end{proof}

\begin{prop}\label{prop explicit comp}
	 Assume we are in the situation of Theorem A. 
	 Let $q$ be a prime different from $19$.
	 Then the $\Z_{(q)}(G\times U)$-lattice
	$$
	\begin{array}{rccl}
	L{(19,q)} &=&& M([(1,0)]_7 \times N_{19} \times U_{19}, 7, q) \\&&\oplus& M([(1,0)]_7 \times \langle ((0,1),1), ((0,0), c_{19}) \rangle , 7, q) \\
&&\oplus& M([(1,0)]_7 \times \langle ((0,1),1), ((0,\beta), c_{19}) \rangle , 7, q)^{\oplus 2}  \\&&\oplus& M([(1,0)]_7 \times \langle ((0,1),1), ((0,2\beta), c_{19}) \rangle , 7, q)^{\oplus 14} \\
&&\oplus& M([(1,0)]_7 \times \langle ((0,1),1), ((0,4\beta), c_{19}) \rangle , 7, q)^{\oplus 3}
	\end{array}
	$$
	is $G$-regular with character $\chi$ as defined in formula \eqref{eqn def chi}.
	Here the $M(X, p,q)$ are as defined in Definition \ref{defi MXpq}.
	
	In the same vein, if $q$ is a prime different from $7$, then the $\Z_{(q)}(G\times U)$-lattice
	$$
		\begin{array}{rccl}
		L{(7,q)} &=&& M([(0,1)]_{19} \times N_7 \times U_7, 19, q)\\&& \oplus& M([(0,1)]_{19} \times \langle ((1,0), 1), ((0,0),c_7) \rangle, 19, q)\\&& \oplus& M([(0,1)]_{19} \times \langle ((1,0),1), ((9\alpha, 0), c_7) \rangle, 19, q)^{\oplus 7}
		\end{array}
	 $$
	is $G$-regular with character $\chi$.
\end{prop}
\begin{proof}
	By Lemma~\ref{lemma extending chin} we have $G$-regular lattices  
	$$ L{(7,q)} =\bigoplus_{\varphi \in {\rm Irr}_{\Q}(N_{7}\times U_{7})^{C_G((0,1))}} M([(0,1)]_{19}\times {\rm Ker}(\varphi), 19, q)^{\oplus \mu(\varphi, (0,1))} $$
	and
	$$ L{(19,q)} =\bigoplus_{\varphi \in {\rm Irr}_{\Q}(N_{19}\times U_{19})^{C_G((1,0))}} M([(1,0)]_{7}\times {\rm Ker}(\varphi), 7, q)^{\oplus \mu(\varphi, (1,0))} $$
	both of which have character $\chi$. We just need to verify that these coincide with the definition of $L(7,q)$ and $L(19,q)$ made in the statement of the proposition.
	The $\varphi$ over which these direct sums range were described in Lemma~\ref{lemma phi reps} and the $\mu(\varphi, (1,0))$ and $\mu(\varphi, (0,1))$ can be computed using their definition in Lemma~\ref{lemma xi proper}. We will now do this explicitly. 
	
	By Lemma~\ref{lemma phi reps} there are $3+d = 6$ elements in ${\rm Irr}_{\Q}(N_{7}\times U_{7})/{C_G((0,1))}$ and ${\rm Irr}_{\Q}(N_{19}\times U_{19})/{C_G((1,0))}$. For $p \in  \{7,19\}$ define the following characters of $N_p \times U_p$, which are representatives of the classes of interest: $1_p$ is the trivial character, $\eta_p$ a non-trivial character with kernel $N_p$, $\psi_p$ a non-trivial character such that $U_p$ is in the kernel of $\psi_p$ and $\varphi_{p,i}$ a non-trivial character such that the kernel of $\varphi_{p,i}$ is $\langle ((1,0),1), ((\alpha^{8\cdot i}\cdot \alpha, 0), c_7) \rangle$ and $ \langle ((0,1),1), ((0,\beta^{20\cdot i}\cdot \beta), c_{19}) \rangle$ respectively, where $0 \leq i \leq 2$. This shows that the kernels of the various $\varphi$'s are as claimed, and it remains to compute the $\mu(\varphi,n)$'s.
	
	For convenience let $n = (0,1)$ for $p=7$ and $n = (1,0)$ for $p =19$. Then by Lemma~\ref{lemma xi proper} we know $\mu(1_p,n) =1$, $\mu(\eta_p, n) = 0$ and $\mu(\psi_p,n) = 1$, for either $p$. Recall that $\alpha^8 = 3$ and $\beta^{20} = 2$. By Lemma~\ref{lemma xi proper} we need to determine elements $g_1, g_2, h_1, h_2 \in G$ such that 
	$$(3\alpha+\mathbb{F}_7)^{g_1} = \alpha + \mathbb{F}_7, (9\alpha+\mathbb{F}_7)^{g_2} = \alpha + \mathbb{F}_7, (2\beta+\mathbb{F}_{19})^{h_1} = \beta + \mathbb{F}_{19} \  \text{and} \ (4\beta+\mathbb{F}_{19})^{h_2} = \beta + \mathbb{F}_{19}.$$
	Since all of these elements must stabilise $\mathbb{F}_7$ or $\mathbb{F}_{19}$, respectively, we get $g_1,g_2 \in \langle c^{8} \rangle$ (the subgroup of $\langle c \rangle$ corresponding to multiplication by elements of $\mathbb{F}_7^\times$) and $h_1, h_2 \in \langle c^{20} \rangle$. Now $c^8$ acts as $\alpha^8 = 3$ on $N_7$ and $c^{20}$ acts as $\beta^{20} = 2$ on $\mathbb{F}_{19}$. We have the following congruences: 
	$$3 \cdot 3^5 \equiv 9 \cdot 3^4 \equiv 1 \bmod 7 \ \ \text{and} \ \ 2 \cdot 2^{17} \equiv 4 \cdot 2^{16} \equiv 1 \bmod 19.$$
	So we can choose $g_1 = c^{5 \cdot 8}$, $g_2 = c^{4 \cdot 8}$, $h_1 = c^{17 \cdot 20}$ and $h_2 = c^{16 \cdot 20}$. This gives 
	$$
	(1,0)^{g_1} = (\alpha^{40},0), (1,0)^{g_2} = (\alpha^{32},0), (0,1)^{h_1} = (0,\beta^{340}) \ \text{and} \ (0,1)^{h_2} = (0,\beta^{320},0)$$ 
	In the notation of Lemma~\ref{lemma xi proper} we get 
	$$
	\begin{array}{rcl}
	\ell(g_1) &=& 40 \equiv 1 \bmod 3 \\ \ell(g_2) &=& 32 \equiv 2  \bmod 3 \\ \ell(h_1) &=& 340 \equiv 1 \bmod 3\\ \ell(h_2) &=& 320 \equiv 2 \bmod 3\end{array}$$ Note that $\ell(1) = 0$. So by Lemma~\ref{lemma xi proper} we obtain 
	$$ \begin{pmatrix} \mu(\varphi_{7,0},n) \\ \mu(\varphi_{7,1},n) \\ \mu(\varphi_{7,2},n) \end{pmatrix} = \begin{pmatrix} r_3(7) & r_1(7) & r_2(7) \\ r_1(7) & r_2(7) & r_3(7) \\ r_2(7) & r_3(7) & r_1(7) \end{pmatrix} \begin{pmatrix}
	 \varepsilon_{(1, 1)^G}\OfU \\ \varepsilon_{(\alpha,1)^G}\OfU \\ \varepsilon_{(\alpha^2,1)^G}\OfU  \end{pmatrix} = \begin{pmatrix} 1 & 2 & 4 \\ 2 & 4 & 1 \\ 4 & 1 & 2 \end{pmatrix} \begin{pmatrix} 2 \\-1 \\ 0  \end{pmatrix} = \begin{pmatrix} 0 \\ 0 \\7 \end{pmatrix}$$
	and
	$$ \begin{pmatrix} \mu(\varphi_{19,0},n) \\ \mu(\varphi_{19,1},n) \\ \mu(\varphi_{19,2},n) \end{pmatrix} = \begin{pmatrix} r_3(19) & r_1(19) & r_2(19) \\ r_1(19) & r_2(19) & r_3(19) \\ r_2(19) & r_3(19) & r_1(19) \end{pmatrix} \begin{pmatrix} \varepsilon_{(1, 1)^G}\OfU \\ \varepsilon_{(1,\beta)^G}\OfU \\ \varepsilon_{(1,\beta^2)^G}\OfU  \end{pmatrix} = \begin{pmatrix} 4 & 9 & 6 \\ 9 & 6 & 4 \\ 6 & 4 & 9 \end{pmatrix} \begin{pmatrix} 2\\ 0 \\ -1  \end{pmatrix} = \begin{pmatrix}  2 \\ 14 \\ 3	 \end{pmatrix}$$
	This shows that the $\mu(\varphi, n)$'s are as claimed, which concludes the proof.
\end{proof}

\section*{Acknowledgements}
We would like to thank \'A. del R\'{\i}o, whose research helped pave the way for this counterexample.
We would also like to express our gratitude to Rob who helped start this collaboration.

\bibliographystyle{alpha}
\bibliography{refs_zc}

\end{document}